\documentclass[11pt]{article}
\usepackage{amsmath, amssymb, amsfonts, amsthm, authblk, color, graphicx, enumitem, mathrsfs, indentfirst}
\usepackage{hyperref, cleveref}
\hypersetup{colorlinks=true, linkcolor=blue, filecolor=magenta, urlcolor=cyan}
\usepackage[textwidth=6in,textheight=9in]{geometry}
\usepackage{float}

\allowdisplaybreaks

\newtheorem{theorem}{Theorem}
\newtheorem{lemma}{Lemma}[section]

\newtheorem{remark}{Remark}[section]

\newtheorem{example}{Example}

\numberwithin{equation}{section}

\newcommand{\keywords}[1]{\small\textbf{\textit{Keywords---}}#1}

\title{Numerical analysis of scattered point measurement-based regularization for backward problems for fractional wave equations}

\author[1]{Dakang Cen}
\author[2]{Zhiyuan Li\footnote{Corresponding author 1: lizhiyuan@nbu.edu.cn, supported by the National Natural Science Foundation of China (no. 12271277), Ningbo Youth Leading Talent Project (no. 2024QL045).  and the Open Research Fund of the Key Laboratory of Nonlinear Analysis \& Applications (Central China Normal University), Ministry of Education, China (no. NAA20230RG002).}}

\author[3]{Wenlong Zhang\footnote{Corresponding author 2: zhangwl@sustech.edu.cn, supported by the National Natural Science Foundation of China under grant
numbers No.12371423 and No.12241104.}}
\affil[1,3]{Department of Mathematics, Southern University of Science and Technology, Shenzhen, 518055, China}
\affil[2]{School of Mathematics and Statistics, Ningbo University, Ningbo, 315211, China}

\begin{document}
\maketitle

\abstract{In this work, our aim is to reconstruct the unknown initial value from terminal data.
We develop a numerical framework on nonuniform time grids for fractional wave equations under the lower regularity assumptions. Then, we introduce a regularization method that effectively handles scattered point measurements contaminated with stochastic noise. The optimal error estimates of stochastic convergence not only balance discretization errors, the noise, and the number of observation points, but also propose an a priori choice of regularization parameters. Finally, several numerical experiments are presented to demonstrate the efficiency and accuracy of the algorithm.}



\section{Introduction}
Assuming that $\alpha\in(1,2)$ and that $\Omega \subset \mathbb{R}^d$, $d=1,2$, is a bounded domain with sufficiently smooth boundary $\partial \Omega$, we consider the following fractional wave equation:
\begin{equation}\label{eq-gov}
\begin{cases}
  \partial_t^\alpha u - \Delta u = f(x,t), & (x,t) \in \Omega \times (0,T), \\
  u(x,0) = a_0(x), & x \in \Omega, \\
  \frac{\partial}{\partial t}u(x,0) = a_1(x), & x \in \Omega, \\
  u(x,t) = 0, & (x,t) \in \partial \Omega \times (0,T),
\end{cases}
\end{equation}
where the operator $\partial_t^\alpha$ is referred to as the Caputo derivative of order $\alpha$, defined by
$$
\partial_t^\alpha \psi(t) := \frac{1}{\Gamma(2 - \alpha)} \int_0^t (t-\tau)^{1-\alpha} \psi''(\tau) d\tau, \quad t>0.
$$
Equation (\ref{eq-gov}) is one of the most famous fractional differential equations. Here, we pay attention to the reconstruction for initial function $a_1(x)$ from observation $u(T)$. This is a continuation of our previous paper \cite{LiZhang2024}.
Specifically, numerical framework for forward problems and stochastic convergence for scattered point measurement-based
regularization are investigated. Without losing generality, let $a_0(x)=0$, $f(x,t)=0$. For nonhomogeneous or $a_0(x) \neq 0$ cases, it can be obtained by simple change of variables.  

Because there are many monographs and papers on this problem, we provide a brief review. There are two types of numerical framework for forward problems: uniform \cite{Jin2016,Jin2016wave,Luo2019} and nonuniform \cite{Stynes2017,LiaoH2018L1,Kopteva2019} time grids. The most representative papers for sub-diffusion equations are: ‘An analysis of the L1 scheme for the subdiffusion equation with non-smooth data’ \cite{Jin2016} and ‘Sharp error estimate of a nonuniform L1 formula for time-fractional reaction subdiffusion equations’ \cite{LiaoH2018L1}. The main ideas of them are based on the discrete Laplace transform and the discrete complementary convolution kernels technology, respectively. As one of the most famous classical numerical methods, L1 schemes for diffusion-wave equations are presented, see Table \ref{REF-method}. $h$ represents the size of the space grids. $N$ is the number of partitions in time grids $0=t_0<\cdots<t_n<\cdots<t_N=T$. And, $\tau=T N^{-1}$ on uniform grids. L1 schemes \cite{Sun2006ANM,Shen2021,LyuP2022SFOR} are proposed based on two kinds of order reduction methods, and they achieve optimal convergence on suitable time grids. The convergence rate of L1 \cite{Shen2021} is better than that in \cite{LyuP2022SFOR,An2022}, but the numerical analysis of it is based on an uncertified assumption. Under the regularity assumptions of $u$, L1 scheme on nonuniform grids with finite difference method and finite element method are presented in \cite{LyuP2022SFOR,Xu2024}. Furthermore, an equivalent integro-differential problem is considered in \cite{Mustapha2013}.  The choice of regularization method in the backward problem is based on the regularity of the initial function. In \cite{Jin2016wave}, two numerical schemes on uniform grids are proposed under the assumption $a_1\in D((-\Delta)^{\gamma/2})\subset H^{\gamma}(\Omega)$, $\gamma\in[0,2]$, which shows that the initial singularity may affect the space convergence when $1+(\gamma-2)\alpha/2<0$. To fill the gap of the above results, we consider investigating the L1 method on nonuniform grids under the lower regularity assumptions $a_1\in D((-\Delta)^{q+\epsilon})$, $q\in(\frac{d}{4},1]$ and $0<\epsilon \ll 1$, see Section \ref{frame-forward}.  It shows that taking the suitable observation time $T$, our scheme reaches the optimal convergence rate $O(N^{-(2-\alpha/2)}+h^2)$.

\begin{table}[!ht]\label{REF-method}
\caption{The convergence rates for existing L1 schemes for diffusion-wave equations (\ref{eq-gov}).}
\renewcommand{\arraystretch}{1.25}
\def\temptablewidth{1\textwidth}
\begin{center}
 \begin{tabular*}{\temptablewidth}{@{\extracolsep{\fill}}lccc}\hline
 scheme       & rate              &time grids      &regularity assumption     \\ \hline
 L1 \cite{Sun2006ANM}          & $O(\tau^{3-\alpha}+h^2)$       &uniform        &$u(x,t)\in C^{(4,3)}_{x,t}(\Omega\times[0,T])$                   \\ \hline
L1 \cite{Zhang2012SIAM}          & $O(\tau^{3-\alpha}+h^4)$       &uniform        &$u(x,t)\in C^{(6,3)}_{x,t}(\Omega\times[0,T])$                   \\ \hline
L1 \cite{Shen2021}            & $O(N^{-(3-\alpha)}+h^2)$    &nonuniform        &$a_1\in D((-\Delta)^{3})$                   \\ \hline
L1 \cite{LyuP2022SFOR}         & $O(N^{-(2-\alpha/2)}+h^2)$    &nonuniform        &$u(x,\cdot)\in H^{4}(\Omega)$                   \\ \hline
L1 \cite{An2022}              & $O(N^{-(2-\alpha/2)}+h)$     &nonuniform    &$a_1\in D((-\Delta)^{2})$                  \\ \hline
Our L1  & $O(N^{-(2-\alpha/2)}+t_n^{1-\alpha(1-q-\epsilon)}h^2)$ & nonuniform & $a_1\in D((-\Delta)^{q+\epsilon})$ \\ \hline
\end{tabular*}
\end{center}
\end{table}

It is well known that the Mittag-Leffler functions $E_{\alpha,\beta}(z)=\sum_{k=0}^\infty\frac{z^k}{\Gamma{(\alpha k+\beta)}}$, $z\in \mathbb{C}$, $\alpha\in(0,1)\cup(1,2)$, $\beta\in \mathbb{C}$, play an important role in investigating the behavior of the solution for fractional differential equations. The potential existence of real roots from the Mittag-Leffler functions in the case $\alpha\in(1,2)$ makes the solution to backward problems for fractional wave equations non-unique. One must make additional assumptions on the terminal time, initial value, or observation data \cite{Wen2023solving,WenJ2023backward,Zhang2022backward,Wei2018backward,floridia2020backward}. 
Notably, when $\alpha$ is in $(1,\frac43]$, additional constraints on
the terminal time $T$ are not required, which relaxes the conditions for the stability
of the backward problem in our previous work. 
A Tikhonov regularization method based on scattered observations was proposed. Despite the presence of large observation errors, we can still obtain more precise inversion results by increasing the number of observation points $n$, which is difficult for classical regularization algorithms to achieve. Let $a_1^\sigma$ be the solution found by regularization methods, where $\sigma$ represents the noise level. Taking the optimal regularization parameters $\rho$, the estimates $Err=\|a_1-a_1^\sigma\|_{L^2}$ are presented in the following table. 

\begin{table}[!ht]
\caption{The convergence rates for existing regularization methods for backward diffusion-wave equations.}
\renewcommand{\arraystretch}{1.25}
\def\temptablewidth{1\textwidth}
\begin{center}
 \begin{tabular*}{\temptablewidth}{@{\extracolsep{\fill}}lccc}\hline
 method       &optimal estimate                    &regularity assumption     \\ \hline
 Quasi-reversibility \cite{Wen2023solving}          & $Err\rightarrow0$ as $\sigma\rightarrow0$ ($\frac{\rho(\sigma)}{\sigma}\rightarrow0$) &$a_1\in L^2(\Omega)$ \\ \hline
 Quasi-reversibility \cite{WenJ2023backward}          & $O(\sigma^{\frac12})$ &$a_1\in H_0^2(\Omega)$ \\ \hline
Tikhonov \cite{Wei2018backward}          & $O(\sigma^{\frac{\gamma}{\gamma+2}})$ &$a_1\in D((-\Delta)^{\gamma/2})$ \\ \hline
Quasi-boundary \cite{Zhang2022backward}          & $O(\sigma^{\frac{\gamma}{\gamma+2}})$ &$a_1\in D((-\Delta)^{\gamma/2})$ \\ \hline             
Our method & $O\bigg((\sigma n^{-\frac{1}{2}})^{\frac{8(1+q+\epsilon)}{4(1+q+\epsilon)+d}} +n^{-\frac{4(1+q+\epsilon)}{d}}\bigg)^{\frac{q+\epsilon}{2(1+q+\epsilon)}}$  & $a_1\in D((-\Delta)^{q+\epsilon})$ \\ \hline
\end{tabular*}
\end{center}
\end{table}

Our previous work focuses on the theoretical analysis of backward problems. As a continuation, we consider such problems in the numerical framework. Our goal is to give an answer to the question: \textit{Is it possible to derive an a priori error estimate, showing the way to balance discretization error, the noise, the regularization parameter, and the number of observation points?}
Specifically, our innovation points are as follows.
\begin{enumerate}
    \item The optimal error estimates not only balance discretization error, the noise, and the number of observation points, but also propose an a priori choice of regularization parameters.

    \item To our best knowledge, it is the first work considering numerical framework on nonuniform grids under the lower regularity assumptions. We also propose optimal choices of mesh parameter $r_{opt}$ for different application cases on graded meshes $t_n=T(n/N)^r$, $r\ge1$. Specifically, $r_{opt}=2$, $1+\frac{2}{2-\alpha}$ for $\Delta a_1\in L^2(\Omega)$ and $a_1\in D((-\Delta)^{\gamma})$, $\gamma\in(\frac{d}{4},1)$, respectively.
    
    \item In our previous work \cite{LiZhang2024}, there are two cases of observation time $T$ for the stability of the backward problem to the fractional wave equations when $\alpha$ is in $(1, \frac{4}{3}]$ or $(\frac{4}{3},2)$. In our numerical framework, we find that optimal error estimates can be obtained under the same strategy of observation time $T$.
    
\end{enumerate} 

The structure of this paper is as follows. In Section \ref{frame-forward}, a L1 numerical framework is proposed for the forward problem. Convergence of it is presented under the lower regularity assumptions. The strategy of observation time $T$ for the backward problem is given based on theoretical results. In Section \ref{sec-regularization}, we introduce a scattered point measurement-based regularization method and derive the optimal error estimates of stochastic convergence under the numerical framework. Regularity assumptions used in Section \ref{frame-forward} are confirmed in Section \ref{sec-reg}. Numerical experiments are carried out to verify the theoretical results in Section \ref{sec-num}.

\section{Numerical framework for forward problems}\label{frame-forward}
Actually, weak singularity of the solution has become an important subject in numerical analysis for fractional diffusion wave equations. Recently, a novel order reduction method \cite{LyuP2022SFOR}, called symmetric fractional order reduction(SFOR), was proposed such that the analysis techniques \cite{LiaoH2018L1,LiaoH2021L2_1} on temporal nonuniform mesh work successfully. Although this method has been applied to numerically solve multiple problems, its feasibility has not been fully considered. 

\subsection{Feasibility of the SFOR method}
The main idea is presented in Lemma \ref{lem-sym}. It has certain requirements for the smoothness of the solution.
\begin{lemma}\label{lem-sym}
For $\alpha\in(1,2)$ and $u(t)\in C^1([0,T])\cap C^2((0,T])$, we have
\[
\partial_t^\alpha u(t)=\partial_t^{\frac{\alpha}{2}}(\partial_t^{\frac{\alpha}{2}}u(t))-u'(0)\omega_{2-\alpha}(t),
\]
where $\omega_{p}(t)=\frac{t^{p-1}}{\Gamma{(p)}}$, $p\geq0$.
\end{lemma}
The proof is given in \cite{LyuP2022SFOR}. Here, we just focus on the key step of it as follows
\begin{align*}
\partial_t^{\frac{\alpha}{2}}u(t)
&=\int_0^t\omega_{1-\frac{\alpha}{2}}(t-s)u'(s)ds\\
&=-u'(s)\omega_{2-\frac{\alpha}{2}}(t-s)|_0^t+\int_0^t\omega_{2-\frac{\alpha}{2}}(t-s)u''(s)ds\\
&=u'(0)\omega_{2-\frac{\alpha}{2}}(t)+\int_0^t\omega_{2-\frac{\alpha}{2}}(t-s)u''(s)ds.
\end{align*}
Obviously, the integral $\int_0^t\omega_{2-\frac{\alpha}{2}}(t-s)u''(s)ds$ exists
when $|u''(t)|\le C(1+t^\gamma)$, $\gamma> -1$. It implies that the SFOR method works under the conditions $a_0=0$ and $a_1\in D((-\Delta)^{\gamma+\epsilon})$, $\gamma\in(\frac{d}{4},1]$ and $0<\epsilon \ll 1$, see Theorems \ref{thm-reg-maxi-1} and \ref{thm-reg-maxi-2}. And, $|v(0)|=0$ is derived from $u(t)\in C^1([0,T])$. Let $\nu=\alpha/2$, model (\ref{eq-gov}) is transformed into the following form
\begin{equation}\label{eq-gov-trans}
\begin{cases}
  \partial_t^\nu v - \Delta u = a_1(x)\omega_{2-\alpha}(t), & (x,t) \in \Omega \times (0,T), \\
  v  = \partial_t^\nu u, & (x,t) \in \Omega \times (0,T), \\
  u(x,0) = v(x,0)=0, & x \in \Omega, \\
  u(x,t) = v(x,t)=0, & (x,t) \in \partial \Omega \times (0,T).
\end{cases}
\end{equation}
In the next subsections, a semi-discrete time scheme is presented for Eq. (\ref{eq-gov-trans}). The stability and convergence of them are derived under some reasonable regularity assumptions. For spatial discretization, we adopt standard Galerkin method with continuous piecewise linear finite elements. Let $u_h$ be the space semi-discrete solution with a mesh size $h$. Following the idea in \cite{Jin2016wave}, the spatial error $u-u_h$ is presented in Theorem \ref{thm-semi-space}. When $\alpha(1-\gamma-\epsilon)<1$, prefactor $t^{1-\alpha(1-\gamma-\epsilon)}\rightarrow\infty$, $t\rightarrow0$, which reflects the initial singularity. 
\begin{theorem}\label{thm-semi-space}
If $a_0=0$, $a_1\in D((-\Delta)^{\gamma+\epsilon})$, $\gamma\in(\frac{d}{4},1]$ and $0<\epsilon \ll 1$, then
\[
\|u(t)-u_h(t)\|_{L^2}+h\|\nabla(u(t)-u_h(t))\|_{L^2} \leq Ch^2t^{1-\alpha(1-\gamma-\epsilon)}\|(-\Delta)^{\gamma+\epsilon}a_1\|_{L^2}.
\] 
\end{theorem}
Specifically, the estimate is consistent in global time when $d=2$, that is $\|u(t)-u_h(t)\|_{L^2}\sim h^2$.
In this case, graded temporal meshes can effectively resolve the impaction of weak singularity. The convergence analysis is prposed in subsection \ref{con-2D}.

\subsection{Time semi-discrete scheme}
Here, the L1 formula (\ref{L1}), one of the most classical discrete method, is used to approximate the Caputo derivative at nonuniform meshes $\{t_n|0=t_0<t_1<\dots <t_N=T\}$. Denote $\tau_n=t_n-t_{n-1}$, $n\geq1$.
\begin{align}\nonumber
\bar\partial_t^\nu v(t_n)
:&= \sum_{k=1}^{n}\int_{t_{k-1}}^{t_k}\omega_{1-\nu}(t_n-s)\frac{v(t_k)-v(t_{k-1})}{\tau_k}ds\\ \label{L1}
&=\sum_{k=1}^nA_{n-k}^{(n)}\nabla_\tau v(t_k),
\end{align}
where $\nabla_\tau v(t_k)=v(t_k)-v(t_{k-1})$, $A_{n-k}^{(n)}:=\int_{t_{k-1}}^{t_k}\frac{\omega_{1-\nu}(t_n-s)}{\tau_k}ds$. The corresponding numerical scheme (\ref{num-scheme})  is as following:
\begin{equation}\label{num-scheme}
\begin{cases}
  \bar\partial_t^\nu V^n - \Delta U^n = a_1(x)\omega_{2-\alpha}(t_n), & 1\le n \le N, \\
  V^n  = \bar\partial_t^\nu U^n, & 1\le n \le N, \\
  U(x,0) = V(x,0)=0, & x \in \Omega, \\
  U(x,t) = V(x,t)=0, & (x,t) \in \partial \Omega \times (0,T),
\end{cases}
\end{equation}
where $U$ and $V$ are numerical solutions corresponding to $u$ and $v$ in (\ref{eq-gov-trans}). Some important lemmas are introduced.

\begin{lemma}\label{FD-ineq}
(\cite{LiaoH2018L1})For $V^n$, $0\leq n\leq N$, one has
\begin{align*}
(\bar\partial_t^\nu V^n,V^n)\geq \frac{1}{2}\bar\partial_t^\nu \|V^n\|_{L^2}^2.    
\end{align*}
\end{lemma}

\begin{lemma}\label{grownwall}
(\cite{LyuP2022SFOR})Let $(g^n)_{n=1}^N$ and $(\lambda_l)_{l=0}^{N-1}$ be given nonnegative sequences. Assume that there exists a constant $\Lambda$ such that $\Lambda\geq\sum_{l=0}^{N-1}\lambda_l$, and that the maximum step satisfies 
$$\max_{1\leq n \leq N}\tau_n\leq \frac{1}{\sqrt[\nu]{\Gamma{(2-\nu)}\Lambda}}.$$
Then, for any nonnegative sequence $(v^k)_{k=0}^N$ and $(w^k)_{k=0}^N$ satisfying
$$\sum_{k=1}^nA_{n-k}^{(n)}\nabla_\tau\big[(v^k)^2+(w^k)^2\big]\leq\sum_{k=1}^n\lambda_{n-k}\big(v^k+w^k\big)^2+(v^n+w^n)g^n,~~1\leq n\leq N,$$
it holds that
$$v^n+w^n\leq 4E_\nu(4\Lambda t_n^\nu)\bigg(v^0+w^0+\max_{1\leq k\leq n}\sum_{j=1}^kP_{k-j}^{(k)}g^j\bigg),~~1\leq n\leq N,$$
where $E_\nu(z)=\sum_{k=0}^\infty\frac{z^k}{\Gamma{(1+k\nu)}}$ is the Mittag-Leffler function.
\end{lemma}

\begin{lemma}\label{P}
For the sequence $(P_{n-j}^{(n)})_{j=1}^n$, some properties are given in \cite{LiaoH2018L1}.
\begin{align*}
&\sum_{j=k}^nP_{n-j}^{(n)}A_{j-k}^{(j)}\equiv1,~~1\leq k\leq n,\\
&0\leq P_{n-j}^{(n)}\leq \Gamma{(2-\nu)}\tau_j^\nu,~~\sum_{j=1}^nP_{n-j}^{(n)}\omega_{1-\nu}(t_j)\leq C,~~ 1\leq j\leq n\leq N.  
\end{align*}
\end{lemma}

\subsection{Stability}
\begin{theorem}\label{thm-stab-1}
If $a_0=0$, $a_1\in D((-\Delta)^{\gamma+\epsilon})$, $\gamma\in(\frac{d}{4},1]$ and $0<\epsilon \ll 1$, then
\[
\|V^n\|_{L^2}+\|\nabla U^n\|_{L^2} \leq C\bigg(\|V^0\|_{L^2}+\|\nabla U^0\|_{L^2}+2\|a_1\|_{L^2}\max_{1\le k\le n}\sum_{j=1}^kP_{k-j}^{(k)}\omega_{2-\alpha}(t_j)\bigg),~~1\le n\le N.
\]
\end{theorem}
\begin{proof}
Taking the inner product with $V^n$ and $\Delta U^n$ for the first two equations of (\ref{eq-gov-trans}), respectively. It gives that
\begin{align*}
(\bar\partial_t^\nu V^n,V^n) - (\Delta U^n,V^n) &= \omega_{2-\alpha}(t_n)(a_1,V^n), \\
(V^n,\Delta U^n)  &= (\bar\partial_t^\nu U^n,\Delta U^n). 
\end{align*}
Adding above equations, it comes that
\begin{align*}
(\bar\partial_t^\nu V^n,V^n) + (\bar\partial_t^\nu \nabla U^n,\nabla U^n) = \omega_{2-\alpha}(t_n)(a_1,V^n).
\end{align*}
By Lemma \ref{FD-ineq}, one has
\begin{align*}
\bar\partial_t^\nu (\|V^n\|_{L^2}^2 + \|\nabla U^n\|_{L^2}^2) 
\le& 2\omega_{2-\alpha}(t_n)\|a_1\|_{L^2}\|V^n\|_{L^2}\\
\le& 2\omega_{2-\alpha}(t_n)\|a_1\|_{L^2}(\|V^n\|_{L^2}+\|\nabla U^n\|_{L^2}).
\end{align*}
The desired result follows from Lemma \ref{grownwall}.
\end{proof}

\begin{remark}
If $d=1$ in Theorem \ref{thm-stab-1}, then the $L^\infty$ stability is derived by the embedding inequality
\[
\|U^n\|_{L^\infty} \leq C\bigg(\|V^0\|_{L^2}+\|\nabla U^0\|_{L^2}+2\|a_1\|_{L^2}\max_{1\le k\le n}\sum_{j=1}^kP_{k-j}^{(k)}\omega_{2-\alpha}(t_j)\bigg),~~1\le n\le N.
\]
\end{remark}

\begin{theorem}\label{thm-stab-2}
If $a_0=0$ and $a_1\in D((-\Delta)^{\gamma+\epsilon})$, $\gamma\in(\frac{1}{2},1]$ and $0<\epsilon \ll 1$, then
\[
\|\nabla V^n\|_{L^2}+\|\Delta U^n\|_{L^2} \leq C\bigg(\|\nabla V^0\|_{L^2}+\|\Delta U^0\|_{L^2}+2\|\nabla a_1\|_{L^2}\max_{1\le k\le n}\sum_{j=1}^kP_{k-j}^{(k)}\omega_{2-\alpha}(t_j)\bigg)
\]
is valid for any $1\le n\le N$. Here the constant $C>0$ is independent of $n$.
\end{theorem}
\begin{proof}
Taking the inner product with $-\Delta V^n$ and $-\Delta^2 U^n$ for the first two equations of (\ref{eq-gov-trans}), respectively. It gives that
\begin{align*}
-(\bar\partial_t^\nu V^n,\Delta V^n) + (\Delta U^n,\Delta V^n) &= -\omega_{2-\alpha}(t_n)(a_1,\Delta V^n), \\
-(V^n,\Delta^2 U^n)  &= -(\bar\partial_t^\nu U^n,\Delta^2 U^n). 
\end{align*}
Adding above equations, it comes that
\begin{align*}
(\bar\partial_t^\nu \nabla V^n,\nabla V^n) + (\bar\partial_t^\nu \Delta U^n,\Delta U^n) = \omega_{2-\alpha}(t_n)(\nabla a_1,\nabla V^n).
\end{align*}
By Lemma \ref{FD-ineq}, one has
\begin{align*}
\bar\partial_t^\nu (\|\nabla V^n\|_{L^2}^2 + \|\Delta U^n\|_{L^2}^2) 
&\le 2\omega_{2-\alpha}(t_n)\|\nabla a_1\|_{L^2}\|\nabla V^n\|_{L^2}\\
&\le 2\omega_{2-\alpha}(t_n)\|\nabla a_1\|_{L^2}(\|\nabla V^n\|_{L^2}+\|\Delta U^n\|_{L^2}).
\end{align*}
The desired result follows from Lemma \ref{grownwall}.
\end{proof}

\begin{remark}
When $d=2$, $a_0=0$, $a_1\in D((-\Delta)^{\gamma+\epsilon})$, $\gamma\in(\frac{d}{4},1]$  and $0<\epsilon \ll 1$, from Theorem \ref{thm-stab-2} and the embedding inequality, one gets 
\[
\|U^n\|_{L^\infty} \leq C\bigg(\|\nabla V^0\|_{L^2}+\|\Delta U^0\|_{L^2}+2\|\nabla a_1\|_{L^2}\max_{1\le k\le n}\sum_{j=1}^kP_{k-j}^{(k)}\omega_{2-\alpha}(t_j)\bigg),~~1\le n\le N.
\]
\end{remark}

\subsection{Convergence of the case in $\mathbb{R}^2$}\label{con-2D}
In this part, we propose a convergence estimate for the numerical scheme (\ref{num-scheme}). Our main focus is on the bounded domain $\Omega\subset\mathbb{R}^2$.  The convergence analysis is considered under the following conditions: $a_0=0$ and $a_1\in D((-\Delta)^{\gamma+\epsilon})$, $\gamma\in(\frac{1}{2},1]$, $0<\epsilon \ll 1$. This is consistent with the feasibility conditions for the SFOR framework. From Theorem \ref{thm-semi-space}, we know that the spatial error is consistent in global time for the two-dimensional situation. Graded temporal meshes can be used directly to resolve the initial singularity. However, the case in $\mathbb{R}^1$ is more complicated. The choice of time step size may affect the spatial error as $t\rightarrow0$ if $a_0=0$ and $a_1\in D((-\Delta)^{\gamma+\epsilon})$, $\gamma\in(\frac{1}{4},1]$, $0<\epsilon \ll 1$. Therefore, convergence is discussed in two subsections \ref{con-2D}-\ref{con-1D}. 

Eq. (\ref{eq-gov-trans}) becomes the following form at $t_n=T(n/N)^r$, $r\geq1$, $0\le n\le N$. It is easy to verify that $\tau_n\leq CN^{-1}t_n^{1-\frac{1}{r}}$, $n\geq2$. Denote $r_1^n:=\bar\partial_t^\nu v^n-\partial_t^\nu v^n$ and $r_2^n:=\partial_t^\nu u^n-\bar\partial_t^\nu u^n$.
\begin{equation}\nonumber
\begin{cases}
  \bar\partial_t^\nu v^n - \Delta u^n = a_1(x)\omega_{2-\alpha}(t_n)+r_1^n, & 1\le n\le N, \\
  v^n  = \bar\partial_t^\nu u^n + r_2^n, & 1\le n\le N, \\
  u(x,0) = v(x,0)=0, & x \in \Omega, \\
  u(x,t) = v(x,t)=0, & (x,t) \in \partial \Omega \times (0,T).
\end{cases}
\end{equation}
Denote $\bar u:=u^n-U^n$ and $\bar v:=v^n-V^n$. One has the error system (\ref{error-system}):
\begin{equation}\label{error-system}
\begin{cases}
  \bar\partial_t^\nu \bar v^n - \Delta \bar u^n = r_1^n, & 1\le n\le N, \\
  \bar v^n  = \bar\partial_t^\nu \bar u^n + r_2^n, & 1\le n\le N, \\
  \bar u(x,0) = \bar v(x,0)=0, & x \in \Omega, \\
  \bar u(x,t) = \bar v(x,t)=0, & (x,t) \in \partial \Omega \times (0,T).
\end{cases}
\end{equation}

The estimate of $r_1^n$ is based on the regularity result of Theorem \ref{thm-reg-L2-v}.
\begin{lemma}\label{lem-r-1}
For $r_1^n$ in (\ref{error-system}), it holds that
\[
\|r_1^n\| \le Ct_n^{-\nu}N^{-q_1}\|a_1\|_{L^2},~~q_1=\min\{r(1-\nu),2-\nu\}.
\]
\end{lemma}
\begin{proof}
For $n=1$, it holds that
\begin{align*}
\|r_1^1\|_{L^2}
&\le C\bigg\|\int_0^{\tau_1}(\tau_1-s)^{-\nu}\bigg(v'(s)-\frac{v^1-v^0}{\tau_1}\bigg)ds\bigg\|_{L^2}\\
&\le C\bigg\|\int_0^{\tau_1}(\tau_1-s)^{-\nu}\bigg(v'(s)-\frac{1}{\tau_1}\int_0^{\tau_1}v'(y)dy\bigg)ds\bigg\|_{L^2}\\
&\le C\tau_1^{-1}\bigg\|\int_0^{\tau_1}(\tau_1-s)^{-\nu}\int_0^{\tau_1}\big(v'(s)-v'(y)\big)dyds\bigg\|_{L^2}\\
&\le C\tau_1^{-1}\int_0^{\tau_1}(\tau_1-s)^{-\nu}\int_0^{\tau_1}\big(\|v'(s)\|_{L^2}+\|v'(y)\|_{L^2}\big)dyds\\
&\le C\tau_1^{-1}\int_0^{\tau_1}(\tau_1-s)^{-\nu}\int_0^{\tau_1}\big(\|v'(s)\|_{L^2}+\|v'(y)\|_{L^2}\big)dyds\\
&\le C\tau_1^{-1}\int_0^{\tau_1}(\tau_1-s)^{-\nu}\int_0^{\tau_1}\big(s^{-\nu}+y^{-\nu}\big)dyds\\
&\le C\tau_1^{1-2\nu}.
\end{align*}
For the case $n\geq2$, one has
\begin{align*}
\|r_1^n\|_{L^2}
&\le C\bigg\|\sum_{k=0}^{n-1}\int_{t_k}^{t_{k+1}}(t_n-s)^{-\nu}\bigg(v'(s)-\frac{v^{k+1}-v^k}{\tau_{k+1}}\bigg)ds\bigg\|_{L^2}\\
&\le C\sum_{k=0}^{n-1}\bigg\|\int_{t_k}^{t_{k+1}}(t_n-s)^{-\nu}\bigg(v'(s)-\frac{v^{k+1}-v^k}{\tau_{k+1}}\bigg)ds\bigg\|_{L^2}\\
&:=C\sum_{k=0}^{n-1}r_1^{n,k}.
\end{align*}
Then, we consider the estimate of $r_1^n$ term by term. For the case $k=0$, it gives that
\begin{align*}
r_1^{n,0}
&\le C\bigg\|\int_{0}^{t_1}(t_n-s)^{-\nu}\bigg(v'(s)-\frac{v^1-v^0}{\tau_1}\bigg)ds\bigg\|_{L^2}\\
&\le C\int_{0}^{t_1}(t_n-s)^{-\nu}\|v'(s)\|_{L^2}ds+C\tau_1^{-1}\int_0^{t_1}(t_n-s)^{-\nu}\int_0^{t_1}\|v'(y)\|_{L^2}dyds\\
&\le C(t_n-t_1)^{-\nu}\int_0^{t_1}s^{-\nu}ds+C\tau_1^{-1}\int_0^{t_1}(t_n-s)^{-\nu}\int_0^{t_1}y^{-\nu}dyds\\
&\le C(t_n-t_1)^{-\nu}\tau_1^{1-\nu}+C\tau_1^{-\nu}\int_0^{t_1}(t_n-s)^{-\nu}ds\\
&\le C(t_n-t_1)^{-\nu}\tau_1^{1-\nu}\\
&\le Ct_n^{-\nu}\tau_1^{1-\nu},
\end{align*}
where $(t_n-t_1)^{-\nu}=t_n^{-\nu}(1-t_1/t_n)^{-\nu}\le Ct_n^{-\nu}$.

For the case $k= n-1$, following identity is used
\begin{align*}
v'(s)-\frac{v^{n}-v^{n-1}}{\tau_{n}}=\frac{1}{\tau_{n}}\int_{t_{n-1}}^{t_{n}}v'(s)-v'(y)dy=\frac{1}{\tau_{n}}\int_{t_{n-1}}^{t_{n}}\int_{y}^sv''(z)dzdy.
\end{align*}
From Theorem \ref{thm-reg-L2-v}, it holds that
\begin{align*}
\bigg\|v'(s)-\frac{v^{n}-v^{n-1}}{\tau_{n}}\bigg\|_{L^2}
\le \frac{1}{\tau_{n}}\int_{t_{n-1}}^{t_{n}}\int_{\min\{s,y\}}^{\max\{s,y\}}\|v''(z)\|_{L^2}dzdy\le C\tau_{n}\sup_{s\in(t_{n-1},t_n)} s^{-\nu-1}.
\end{align*}
And, one has
\begin{align*}
r_1^{n,n-1}
&\le C\bigg\|\int_{t_{n-1}}^{t_{n}}(t_n-s)^{-\nu}\bigg(v'(s)-\frac{v^{n}-v^{n-1}}{\tau_{n}}\bigg)ds\bigg\|_{L^2}\\
&\le C\tau_{n}\sup_{s\in(t_{n-1},t_n)} s^{-\nu-1}\int_{t_{n-1}}^{t_{n}}(t_n-s)^{-\nu}ds\\
&\le C\tau_n^{2-\nu}\sup_{s\in(t_{n-1},t_n)} s^{-\nu-1}\\
&\le C\bigg(\frac{\tau_n}{t_n}\bigg)^{2-\nu}t_n^{1-2\nu}\\
&\le Ct_n^{-\nu}N^{-q}t_n^{-q/r}t_n^{1-\nu}\\
&\le Ct_n^{-\nu}N^{-q},
\end{align*}
the inequality holds using $\tau_n\leq CN^{-1}t_n^{1-\frac{1}{r}}$.
For the case $1\le k\le n-2$, we have
\begin{align*}
r_1^{n,k}
&\le C\bigg\|\int_{t_{k}}^{t_{k+1}}(t_n-s)^{-\nu}\bigg(v'(s)-\frac{v^{k+1}-v^{k}}{\tau_{k+1}}\bigg)ds\bigg\|_{L^2}\\
&\le C\int_{t_{k}}^{t_{k+1}}(t_n-s)^{-\nu-1}\bigg\|v(s)-\frac{(s-t_k)v^{k+1}+(t_{k+1}-s)v^{k}}{\tau_{k+1}}\bigg\|_{L^2}ds\\
&\le C\tau_{k+1}^2\sup_{s\in(t_k,t_{k+1})}s^{-\nu-1}\int_{t_{k}}^{t_{k+1}}(t_n-s)^{-\nu-1}ds\\
&\le C\bigg(\tau_{k+1}^{2-\nu}t_{k+1}^\nu\sup_{s\in(t_k,t_{k+1})}s^{-\nu-1}\bigg)\bigg(\tau_{k+1}^\nu t_{k+1}^{-\nu}\int_{t_k}^{t_{k+1}}(t_n-s)^{-\nu-1}ds\bigg)\\
&\le C\bigg(\tau_{k+1}^{2-\nu}t_{k+1}^\nu\sup_{s\in(t_k,t_{k+1})}s^{-\nu-1}\bigg)\bigg(\tau_{k+1}^\nu \int_{t_k}^{t_{k+1}}s^{-\nu}(t_n-s)^{-\nu-1}ds\bigg).
\end{align*}
Denote $\Phi_k:=\tau_{k+1}^{2-\nu}t_{k+1}^\nu\sup_{s\in(t_k,t_{k+1})}s^{-\nu-1}\sim N^{-q}$. Hence, we obtain
\begin{align*}
\sum_{k=1}^{n-2}r_1^{n,k}
&\le C\sum_{k=1}^{n-2}\Phi_k\tau_{k+1}^\nu \int_{t_k}^{t_{k+1}}s^{-\nu}(t_n-s)^{-\nu-1}ds\\
&\le C\max_k\Phi_k\bigg(\tau_{n+1}^\nu \int_{t_1}^{t_{n-1}}s^{-\nu}(t_n-s)^{-\nu-1}ds\bigg)\\
&\le Ct_n^{-\nu}N^{-q},
\end{align*}
the last inequality holds by 
\begin{align*}
\int_{t_1}^{t_{n-1}}s^{-\nu}(t_n-s)^{-\nu-1}ds
&\le \int_{t_1}^{t_{n}/2}s^{-\nu}(t_n-s)^{-\nu-1}ds+\int_{t_{n}/2}^{t_{n-1}}s^{-\nu}(t_n-s)^{-\nu-1}ds\\
&\le C(t_n/2)^{1-\nu}(t_n-t_n/2)^{-\nu-1}+C(t_n/2)^{-\nu}\tau_n^{-\nu}\\
&\le Ct_n^{-\nu}\tau_n^{-\nu}+Ct_n^{-\nu}\tau_n^{-\nu}\\
&\le Ct_n^{-\nu}\tau_{n+1}^{-\nu}.
\end{align*}
The proof completes based on above estimates.
\end{proof}

Following the idea in Lemma \ref{lem-r-1}, if $a_1\in D((-\Delta)^{\gamma+\epsilon})$, $\gamma\in(\frac12,1]$, $0<\epsilon \ll 1$, from the regularity of $u$ in Theorem \ref{thm-reg-delta-1}, we get the estimate for $r_2^n$ in Lemma \ref{lem-r-2}.
\begin{lemma}\label{lem-r-2}
For $r_2^n$ in (\ref{error-system}), it holds that
\[
\|\nabla r_2^n\| \le Ct_n^{-\nu}(\tau_1+N^{-q_2})\|(-\Delta)^{\gamma+\epsilon}a_1\|_{L^2},~~q_2=\min\{r(1+\epsilon\alpha),2-\nu\}.
\]
\end{lemma}

\begin{theorem}\label{thm-conv-1}
If $a_0=0$, $a_1\in D((-\Delta)^{\gamma+\epsilon})$, $\gamma\in(\frac{1}{2},1]$ and $0<\epsilon \ll 1$, then
\[
\|\bar v^n\|_{L^2}+\|\nabla \bar u^n\|_{L^2} \leq C(\tau_1+N^{-q_1}+N^{-q_2})\|(-\Delta)^{\gamma+\epsilon}a_1\|_{L^2},~~1\le n\le N.
\]
\end{theorem}
\begin{proof}
Taking the inner product with $\bar v^n$ and $\Delta \bar u^n$ for the first two equations of (\ref{error-system}), respectively. It gives that
\begin{align*}
(\bar\partial_t^\nu \bar v^n,\bar v^n) - (\Delta \bar u^n,\bar v^n) &= (r_1^n,\bar v^n), \\
(\bar v^n,\Delta \bar u^n)  &= (\bar\partial_t^\nu \bar u^n,\Delta \bar u^n)+(r_2^n,\Delta \bar u^n). 
\end{align*}
Adding above equations, it comes that
\begin{align*}
(\bar\partial_t^\nu \bar v^n,\bar v^n) + (\bar\partial_t^\nu \nabla \bar u^n,\nabla \bar u^n) = (r_1^n,\bar v^n)+(\nabla r_2^n,\nabla \bar u^n).
\end{align*}
By Lemma \ref{FD-ineq}, one has
\begin{align*}
\frac12\bar\partial_t^\nu (\|\bar v^n\|_{L^2}^2 + \|\nabla \bar u^n\|_{L^2}^2) &\le \|r_1^n\|_{L^2}\|\bar v^n\|_{L^2}+\|\nabla r_2^n\|_{L^2}\|\nabla \bar u^n\|_{L^2}\\
&\le (\|r_1^n\|_{L^2}+\|\nabla r_2^n\|_{L^2})(\|\bar v^n\|_{L^2}+\|\nabla \bar u^n\|_{L^2}).
\end{align*}
The desired result follows from Lemmas \ref{grownwall} and \ref{P}.
\end{proof}

\begin{remark}\label{re-error-R2}
In \cite{JinB2017sub_pod}, the authors presented the analysis framework for sub-diffusion equations. The convergence of the fully discrete scheme of diffusion-wave equations follows analogously. From the results of Theorems \ref{thm-semi-space} and \ref{thm-conv-1}, it implies that the fully consistent $L^2$ norm error reaches optimal $O(h^2+N^{-(2-\nu)})$ in $\Omega\times(0,T]$, when the time grid parameter $r=\frac{2-\nu}{1-\nu}$.
\end{remark}

\subsection{Convergence of the case in $\mathbb{R}^1$}\label{con-1D}
The framework of convergence is the same as that for Theorem \ref{thm-conv-1}. The truncation error $\|r_1^n\|$, $1\leq n\leq N$ is available from Lemma \ref{lem-r-1}. The point is on the estimate $\|\nabla r_2^n\|$, $1\leq n\leq N$. The regularity of $\|\partial_t^m(-\Delta)^{\frac12}u\|_{L^2}$ is derived from Theorem \ref{thm-reg-delta-2}. Then,  one has the following results.
\begin{lemma}\label{lem-r-2-1D}
For $r_1^n$, $r_2^n$ in (\ref{error-system}) with $\Omega\subset \mathbb{R}^1$, it holds that
\[
\|r_1^n\| \le Ct_n^{-\nu}N^{-q_1}\|a_1\|_{L^2},~~q_1=\min\{r(1-\nu),2-\nu\},
\]
\[
\|\nabla r_2^n\| \le Ct_n^{-\nu}N^{-q_3}\|(-\Delta)^{\gamma+\epsilon}a_1\|_{L^2},~~q_3=\min\{r(1-\nu/2),2-\nu\}.
\]
\end{lemma}
The convergence estimate is obtained by Theorem \ref{thm-conv-1} and Lemma \ref{lem-r-2-1D}.
\begin{theorem}\label{thm-conv-2}
If $a_0=0$, $a_1\in D((-\Delta)^{\gamma+\epsilon})$, $\gamma\in(\frac{1}{4},1]$ and $0<\epsilon \ll 1$, then
\[
\|\bar v^n\|_{L^2}+\|\nabla \bar u^n\|_{L^2} \leq C(N^{-q_1}+N^{-q_3})\|(-\Delta)^{\gamma+\epsilon}a_1\|_{L^2},~~1\le n\le N.
\]
\end{theorem}

Combining the results of Theorems \ref{thm-semi-space} and \ref{thm-conv-2}, we obtain the convergence of the fully discrete solution $u_h^n$ at $t_n$, $n\geq1$. The conclusion is different from the one in Remark \ref{re-error-R2}. Specifically, the spatial discrete error is 
$O(h^2)$, $\alpha\in(1,\frac43]$, for $\forall t>0$ in Theorem \ref{thm-conv-3}.
\begin{theorem}\label{thm-conv-3}
If $a_0=0$, $a_1\in D((-\Delta)^{\gamma+\epsilon})$, $\gamma\in(\frac{1}{4},1)$ and $0<\epsilon \ll 1$, then
\[
\| u(t_n)-u_h^n\|_{L^2} \leq C(h^2+N^{-q_1}+N^{-q_3}),~~1<\alpha\leq \frac{1}{1-(\gamma+\epsilon)},
\]
\[
\| u(t_n)-u_h^n\|_{L^2} \leq C(t_n^{1-\alpha(1-\gamma-\epsilon)}h^2+N^{-q_1}+N^{-q_3}),~~ \frac{1}{1-(\gamma+\epsilon)}<\alpha<2.
\]
\end{theorem}

When $\frac{1}{1-(\gamma+\epsilon)}<\alpha<2$, $t_1^{1-\alpha(1-\gamma-\epsilon)}h^2$ may blow up as the time step parameter $r$ is large. The situation becomes better when $t_n$ away from 0. 
\begin{remark}
When $t_n>t^*$, $t^*$ is a fixed positive constant, one has
\[
\| u(t_n)-u_h^n\|_{L^2} \leq C(h^2+N^{-q_1}+N^{-q_3}).
\]
\end{remark}

Combining the above results, we have an interesting finding. First, let us go back to the problem of initial value reconstruction. The stability of the backward problem is as follows. There are some suggestions of the choices of observation time $T$ for $\alpha\in(1,\frac43]$ or $(\frac43,2)$.
\begin{theorem}(\cite{LiZhang2024})\label{th-sta-T}
Letting $\alpha\in(1,2)$, $\beta,\gamma\in [0,1]$ be such that $\gamma+\beta\neq0$, we suppose that the pair $(u,a_1)$ in the space $L^2 \left(0,T; H_0^1(\Omega) \cap H^2(\Omega) \right) \times D((-\Delta)^{\beta})$ is a solution of our backward problem, which corresponds to the measurement data $u(\cdot,T)$. If $\|(-\Delta)^{\beta} a_1\|_{L^2(\Omega)} \le M$ for some positive constant $M$, then the estimate
$$
\|a_1\|_{D((-\Delta)^{-\gamma})}\le  C_T M^{\frac{1-\gamma}{1+\beta}} \left\| u(\cdot, T)\right\|_{L^2(\Omega)}^{\frac{\beta+\gamma}{\beta+1}}.
$$
is valid provided that one of the following conditions hold: 
\begin{enumerate}[label = \roman*.]
    \item $\alpha\in(1,\frac43]$ and $T>0$;
    \item $\alpha\in(\frac43,2)$ and $T^\alpha \not\in \cup_{k=1}^\infty \{ \frac{t_1}{\lambda_k}, \cdots, \frac{t_P}{\lambda_k}\}$.
\end{enumerate}
Here $C_T>0$ is a constant that is independent of $a_1$ and $u$, but may depend on $\alpha,\beta,\gamma,T,M$, and $\Omega$.
\end{theorem}
Then, the choice of $T$ also plays an important role in the stochastic convergence analysis of the regularization method based on the measurements of the scattering points. Naturally, the question arises of whether the choice of $T$ has an effect on the numerical inversion. Remark \ref{Re-pre-numfram} gives a positive conclusion on this. Specifically, the convergence of the numerical method is guaranteed and optimally achieved under a suitable time grid. Therefore, we call it the preserving numerical framework. 
\begin{remark}\label{Re-pre-numfram}(Preserving numerical framework)
If $a_0=0$, $a_1\in D((-\Delta)^{\gamma+\epsilon})$, $\gamma\in(\frac{1}{4},1)$, $0<\epsilon\ll 1$ and $r=1+\frac{2}{2-\alpha}$, then 
\begin{enumerate}[label = \roman*.]
    \item $\|u(t_n)-u_h^n\|_{L^2}\leq C(h^2+N^{-(2-\frac{\alpha}{2})})\|(-\Delta)^{\gamma+\epsilon}a_1\|_{L^2}$, $\alpha\in(1,\frac43]$ and $t_n>0$;
    \item $\|u(t_n)-u_h^n\|_{L^2}\leq C(h^2+N^{-(2-\frac{\alpha}{2})})\|(-\Delta)^{\gamma+\epsilon}a_1\|_{L^2}$, $\alpha\in(\frac43,2)$ and $t_n>t^*$, $t^*$ is a fixed positive constant.
\end{enumerate}
\end{remark}

\section{Scattered point measurement-based regularization}\label{sec-regularization}
For stating the Tikhonov regularization method based on scattered point measurement, we collect a set of  scattered points $ \{ x_i \}_{i = 1}^n$ which are such that $ x_i \neq x_j $ for $ i \neq j $ and are quasi-uniformly distributed in $ \Omega $, that is, there exists a positive constant $ B $ such that $ d_{\max} \le B d_{\min} $ , where $ d_{\max}>0$ and $ d_{\min}>0$ are defined by
$$ 
d_{\max} = \sup_{x \in \Omega} \inf_{1 \leq i \leq n} | x - x_i |,\quad
d_{\min} = \inf_{1 \leq i \neq j \leq n} | x_i - x_j | .
$$
Furthermore, for any $u,v\in C(\overline\Omega)$ and $y\in\mathbb R^n$, we define
$$
(y,v)_n := \frac1n \sum_{i=1}^n y_i v(x_i),\quad (u,v)_n:=\frac1n \sum_{i=1}^n u(x_i)v(x_i),
$$
and the discrete semi-norm 
$$
\|u\|_n := \left(\sum_{i=1}^n \frac{u^2(x_i)}{n} \right)^{\frac12},\quad u\in C(\overline \Omega).
$$

In view of the stability results from the above sections, it follows that the forward operator $S$ is bounded and one-to-one from $X:=D((-\Delta)^\beta)$ to $H^2(\Omega)$. Moreover, let $a^*\in X $ be the unknown initial value of the problem \eqref{eq-gov}. We assume that the measurement data contains noise and is presented in the following form:
\begin{equation}\label{ob-noise0}
m_i = (Sa^*)(x_i) + e_i, \quad i = 1,2,...,n,
\end{equation}
where $\{e_i\}_{i=1}^n$ denote a sequence of random variables that are independent and identically distributed over the probability space. The expectation is such that $\mathbb{E}[e_i] = 0$, and the variances are bounded by $\sigma^2$, that is, $\mathbb{E}[e_i^2] \leq \sigma^2$. 

We denote $\mathbf{x}:=(x_1,x_2,\cdots,x_n)$, $\mathbf{m} = (m_1, m_2, ..., m_n)^T$, and $\mathbf{e}:=(e_1,e_2,\cdots,e_n)$. Then the above term \eqref{ob-noise0} can be rephrased as the vector form:
\begin{equation*}
\mathbf{m} = (Sa^*)(\mathbf{x}) + \mathbf{e}.
\end{equation*}

Let $S_{\tau,h}$ be the fully discrete approximation of the operator $S$. The forward problem is solved by the numerical scheme (\ref{num-scheme}) in finite element space $X_h$. We seek a numerical solution, denoted as $ a_n^*$, for the unknown initial value $a^*$, utilizing the Tikhonov regularization form:
\begin{equation}\label{Tik-vec}
\arg\min_{a \in X_h} \| (S_{\tau,h}a)(\mathbf{x}) - \mathbf{m}\|_n^2 + \rho_n \| a \|^2_{X},
\end{equation}
where $X_h\subset D((-\Delta)^\beta)$ with $\beta\in \mathbb R$ and $ \rho_n > 0 $ is called a regularization parameter.

Here we present our main theorem on the stochastic convergence of the SFOR framework. Estimates $\|S_{\tau,h}a^*-Sa^*\|_{L^2}^2\leq e_S:=O(h^4+N^{-(4-\alpha)})$ have been discussed in Remarks \ref{re-error-R2} and \ref{Re-pre-numfram}. 
The feasibility of the SFOR method suggests $\beta\in(\frac{d}{4},1)$. The observation time $T$ satisfies the requirements in Theorem \ref{th-sta-T}.

\begin{theorem}\label{thm-conv-SFOR}
Let $a_n^* \in X_h$ with dimensions $N_h$ be the unique solution of our Tikhonov regularization form \eqref{Tik-vec}. Then there exist constants $ \rho_0 > 0 $ and $ C > 0 $ such that the following estimates
\[ 
\mathbb{E} \bigl[ \| S_{\tau,h} a_n^* - S a^* \|^2_n \bigr] \leq C(\rho_n+e_S)\| a^* \|^2_X + C\left(1+\frac{e_S}{\rho_n}+\frac{N_he_S}{\rho_n^{1-\frac{d}4/(1+\beta)}}\right)\frac{\sigma^2}{n \rho_n^{\frac{d}4/(1+\beta)}}, 
\]
and
\[ 
\mathbb{E} \bigl[ \| a_n^* - a^* \|^2_{X} \bigr] \leq C\frac{\rho_n+e_S}{\rho_n}\| a^* \|^2_X +C\left(1+\frac{e_S}{\rho_n}+\frac{N_he_S}{\rho_n^{1-\frac{d}4/(1+\beta)}}\right)\frac{\sigma^2}{n \rho_n^{1+\frac{d}4/(1+\beta)}}. 
\]
are valid for any $0< \rho_n \leq \rho_0 $. Here, the constant $C$ is independent of $n,a^*,a_n^*$.
\end{theorem}
\begin{proof}
When the parameter $r$ of graded mesh is taken $1+\frac{2}{2-\alpha}$, it holds that
\begin{align*}
\|S_{\tau,h}a^*-Sa^*\|_{L^2}
\leq C(h^2+N^{-(2-\frac{\alpha}{2})})\|a^*\|_{X},
\end{align*}
details can be found in Remarks \ref{re-error-R2} and \ref{Re-pre-numfram}.
Applying the Lemma 3.10 in \cite{chen2022stochastic} and Lemma 2.4 in \cite{LiZhang2024}, we have
\begin{align*}
\|S_{\tau,h}a^*-Sa^*\|_n
&\leq C\|S_{\tau,h}a^*-Sa^*\|_{L^2}+Ch^2\|Sa^*\|_{H^2}\\
&\leq C\|S_{\tau,h}a^*-Sa^*\|_{L^2}+Ch^2\|a^*\|_{L^2}\\
&\leq C\|S_{\tau,h}a^*-Sa^*\|_{L^2}+Ch^2\|(-\Delta)^{\gamma}a^*\|_{L^2}\\
&\leq C(h^2+N^{-(2-\frac{\alpha}{2})})\|a^*\|_{X}.
\end{align*}
Replace $e(h)$ therein by $e_S$ in Assumption 2.3, the proof is complete following the idea in Theorem 2.10 \cite{chen2022stochastic}.
\end{proof}

\begin{remark}\label{core}
For Theorem \ref{thm-conv-SFOR}, if $e_S\leq C\rho_n$ and $N_he_S\leq C\rho_n^{1-\frac{d}4/(1+\beta)}$, we have
\[ 
\mathbb{E} \bigl[ \| S_{\tau,h} a_n^* - S a^* \|^2_n \bigr] \leq C\rho_n\| a^* \|^2_X + \frac{C\sigma^2}{n \rho_n^{\frac{d}4/(1+\beta)}}, 
\]
and
\[ 
\mathbb{E} \bigl[ \| a_n^* - a^* \|^2_{X} \bigr] \leq C \|a^* \|^2_X +\frac{C\sigma^2}{n \rho_n^{1+\frac{d}4/(1+\beta)}}. 
\]
\end{remark}

Theorem \ref{thm-E-condi} shows that the error estimate is based on key parameters
such as the noise level, the regularization parameter, and the number of observation points.
Then, the optimal regularization parameter is given in Remark \ref{remark-1}.

\begin{lemma}\cite[Theorems 3.3 and 3.4]{utreras1988convergence}
\label{lem-u-un}
There exists a constant  $C>0$  such that for all  $  u \in H^k(\Omega)$ with $k>\frac{d}2$ , the following estimates are valid:
\begin{equation}\label{esti-u<un}
\begin{aligned}
\| u \|^2_{L^2(\Omega)} \leq& C \left( \| u \|^2_n + n^{-\frac{2k}{d}} \| u \|^2_{H^k(\Omega)} \right),\\
\| u \|^2_n \leq& C \left( \| u \|^2_{L^2(\Omega)} + n^{-\frac{2k}{d}} \| u \|^2_{H^k(\Omega)} \right). 
\end{aligned}
\end{equation}
\end{lemma}

\begin{theorem}\label{thm-E-condi}
Suppose that the unknown initial value $a^*\in X=D((-\Delta)^\beta)$ with $\beta\in(\frac{d}{4},1]$, and for the regularization parameter $\rho_n$, if $e_S\leq C\rho_n$ and $N_he_S\leq C\rho_n^{1-\frac{d}4/(1+\beta)}$, then there exists a constant $C=C(\beta,\rho_n,T,\Omega)$ such that
\[
\begin{aligned}
\mathbb{E} \bigl[ \| a_n^* - a^* \|^{2+\frac{2}{\beta}}_{L^2(\Omega)} \bigr] 
\le & C\left[ \| a^* \|^2_{X} + \frac{\sigma^2}{n \rho_n^{1 + \frac{d}4/(1+\beta)}}\right]^{\frac2\beta} \frac1{n^{4(1+\beta)/d}} \\
&+ C\rho_n\left[ \| a^* \|^2_{X} + \frac{\sigma^2}{n \rho_n^{1 + \frac{d}4/(1+\beta)}}\right]^{\frac1\beta+1}\\
&+C\left[ \| a^* \|^2_{X} + \frac{\sigma^2}{n \rho_n^{1 + \frac{d}4/(1+\beta)}}\right]^{\frac1\beta}\left(1+n^{-\frac4d}h^{-2}\right)e_S\|a_n^*\|^2_X.
\end{aligned}
\]
\end{theorem}
\begin{proof}
In view of the estimate \eqref{esti-u<un} in Lemma \ref{lem-u-un}, we see that
\[ 
 \| S_{\tau,h} a_n^* - S a^* \|^2_n  + n^{-\frac4d}  \| S_{\tau,h} a_n^* - S a^* \|^2_{H^2(\Omega)}  \geq C  \| S_{\tau,h} a_n^* - S a^* \|^2_{L^2(\Omega)}.
\]
For $\| S_{\tau,h} a_n^* - S a^* \|^2_{L^2(\Omega)}$, one has
\begin{align*}
\| S_{\tau,h} a_n^* - S a^* \|^2_{L^2(\Omega)}&\ge C\|Sa_n^* - Sa^*\|^2_{L^2(\Omega)}-C\|S_{\tau,h} a_n^* - S a_n^*\|^2_{L^2(\Omega)}\\
&\ge C\|Sa_n^* - Sa^*\|^2_{L^2(\Omega)}-Ce_S\|a_n^*\|^2_X.
\end{align*}
Moreover, from the second estimate in Remark \ref{core} with $X=D((-\Delta)^{\beta})$, $\beta>0$, combined with Theorem \ref{th-sta-T} with $\gamma=0$ further implies that
\[ 
 \| S_{\tau,h} a_n^* - S a^* \|^2_n  + n^{-\frac4d} \| S_{\tau,h} a_n^* - S a^* \|^2_{H^2(\Omega)}+Ce_S\|a_n^*\|^2_X  \geq C  M_n^{-\frac{2}{\beta}}\| a_n^* - a^* \|^{2+\frac{2}{\beta}}_{L^2(\Omega)} .
\]
Here $M_n:= \|(-\Delta)^\beta (a_n^* - a^*)\|_{L^2(\Omega)}$. On the other hand, Lemma 2.4 in \cite{LiZhang2024} implies that
\begin{align*}
 \| S_{\tau,h} a_n^* - S a^* \|^2_{H^2(\Omega)} &\le C\| S_{\tau,h} a_n^* - S a_n^* \|^2_{H^2(\Omega)}+\| S a_n^* - S a^* \|^2_{H^2(\Omega)} \\
&\leq Ch^{-2}e_S\|a_n^*\|_X^2+C \|a_n^* - a^* \|^2_{L^2(\Omega)}.
\end{align*}

Collecting all the above estimates and using the first inequality in Remark \ref{core} with $X=D((-\Delta)^{\beta})$, we see that
\begin{align*}
\mathbb{E} \bigl[ \| a_n^* - a^* \|^{2+\frac{2}{\beta}}_{L^2(\Omega)} \bigr]
\le &C\mathbb{E} \bigl[M_n^{\frac2\beta} \bigr] \left[n^{-\frac4d}\mathbb{E} \bigl[\|a_n^* - a^* \|^2_{L^2(\Omega)} \bigr] +  \rho_n \| a^* \|^2_{X} +  \frac{\sigma^2}{n \rho_n^{\frac{d}{4}/(1+\beta)}}
\right]\\
&+C\mathbb{E}\bigl[M_n^{\frac2\beta} \bigr]\left(1+n^{-\frac4d}h^{-2}\right)e_S\|a_n^*\|^2_X.
\end{align*}

Moreover, for $\varepsilon>0$, by the use of the Young inequality $|\xi\zeta| \le \frac{C}{\varepsilon} |\xi|^p + \varepsilon |\zeta|^q$ with $p=\beta+1$ and $q=\frac{\beta+1}{\beta}$, we obtain
\[
 n^{-\frac4d}\|a_n - a^* \|^2_{L^2(\Omega)}  \le \frac{C}{\varepsilon}n^{-\frac{4(1+\beta)}d} + \varepsilon \|a_n - a^* \|^{2+\frac2\beta}_{L^2(\Omega)}.
\]
Consequently, we see that
\[
\begin{aligned}
&\mathbb{E} \bigl[ \| a_n^* - a^* \|^{2+\frac{2}{\beta}}_{L^2(\Omega)} \bigr] \\
\le&  C\mathbb{E} \bigl[M_n^{\frac2\beta} \bigr] \left[ \frac{1}{\varepsilon}n^{-\frac{4(1+\beta)}d}  + \varepsilon \mathbb{E} \bigl[\|a_n^* - a^* \|^{2+\frac2\beta}_{L^2(\Omega)} \bigr] +  \rho_n \| a^* \|^2_{X} +  \frac{\sigma^2}{n \rho_n^{\frac{d}{4}/(1+\beta)}} \right]\\
&+C\mathbb{E}\bigl[M_n^{\frac2\beta} \bigr]\left(1+n^{-\frac4d}h^{-2}\right)e_S\|a_n^*\|^2_X.
\end{aligned}
\]
By letting $\varepsilon=\frac1{2C\mathbb{E} \bigl[M_n^{\frac2\beta} \bigr]}$, we can see that the term  $\varepsilon\mathbb{E} \bigl[\|a_n^* - a^* \|^{2+\frac2\beta}_{L^2(\Omega)} \bigr]$ on the right hand side of the above inequality can be absorbed, and then we get
\[
\begin{aligned}
\mathbb{E} \bigl[ \| a_n^* - a^* \|^{2+\frac{2}{\beta}}_{L^2(\Omega)} \bigr] 
\le & C\mathbb{E} \bigl[M_n^{\frac2\beta} \bigr] \left[ 2C\mathbb{E} \bigl[M_n^{\frac2\beta} \bigr] n^{-\frac{4(1+\beta)}d}  +  \rho_n \| a^* \|^2_{X} +  \frac{\sigma^2}{n \rho_n^{\frac{d}{4}/(1+\beta)}} \right]\\
&+C\mathbb{E}\bigl[M_n^{\frac2\beta} \bigr]\left(1+n^{-\frac4d}h^{-2}\right)e_S\|a_n^*\|^2_X.
\end{aligned}
\]
Now by using the second estimate in Remark \ref{core} with $X=D((-\Delta)^\beta)$, and noting the definition $M_n:=\|(-\Delta)^\beta (a_n^* - a^*)\|_{L^2(\Omega)}$, it follows that
$$
\mathbb{E} \bigl[M_n^{\frac2\beta} \bigr] \le C\left[ \| a^* \|^2_{X} + \frac{\sigma^2}{n \rho_n^{1 + \frac{d}4/(1+\beta)}}\right]^{\frac1\beta},
$$
from which we further know that
\[
\begin{aligned}
\mathbb{E} \bigl[ \| a_n^* - a^* \|^{2+\frac{2}{\beta}}_{L^2(\Omega)} \bigr] 
\le&  C\mathbb{E} \bigl[M^{\frac2\beta} \bigr] \left[ \frac{\mathbb{E} \bigl[M^{\frac2\beta} \bigr]}{n^{4(1+\beta)/d}}  +  \rho_n \| a^* \|^2_{X} +  \frac{\sigma^2}{n \rho_n^{\frac{d}{4}/(1+\beta)}} \right]\\
&+C\mathbb{E}\bigl[M_n^{\frac2\beta} \bigr]\left(1+n^{-\frac4d}h^{-2}\right)e_S\|a_n^*\|^2_X\\
\le & C\left[ \| a^* \|^2_{X} + \frac{\sigma^2}{n \rho_n^{1 + \frac{d}4/(1+\beta)}}\right]^{\frac2\beta} \frac1{n^{4(1+\beta)/d}} \\
&+ C\rho_n\left[ \| a^* \|^2_{X} + \frac{\sigma^2}{n \rho_n^{1 + \frac{d}4/(1+\beta)}}\right]^{\frac1\beta+1}\\
&+C\left[ \| a^* \|^2_{X} + \frac{\sigma^2}{n \rho_n^{1 + \frac{d}4/(1+\beta)}}\right]^{\frac1\beta}\left(1+n^{-\frac4d}h^{-2}\right)e_S\|a_n^*\|^2_X.
\end{aligned}
\]
We can complete the proof of the theorem.
\end{proof}

\begin{remark}\label{remark-1}
Remark \ref{core} and Theorem \ref{thm-E-condi} indicate that the optimal regularization parameter has the form: $\rho_n^{\frac12 + \frac{d}8/(1+\beta)} = O(\sigma n^{-\frac12} \|a^*\|^{-1}_{X})$ for $\beta\in(\frac{d}{4},1]$, and the optimal error will be
\[
\mathbb{E} \bigl[ \| a_n^* - a^* \|^{2+\frac2\beta}_{L^2(\Omega)} \bigr]\le C \left(\rho_n+ \frac1{n^{4(1+\beta)/d}} \right)+C\left(1+n^{-\frac4d}h^{-2}\right)e_S\|a_n^*\|^2_X.
\]
If $\left(1+n^{-\frac4d}h^{-2}\right)e_S\|a_n^*\|^2_X<\rho_n+ \frac1{n^{4(1+\beta)/d}}$, we have
\[
\mathbb{E} \bigl[ \| a_n^* - a^* \|^{2+\frac2\beta}_{L^2(\Omega)} \bigr]\le C \left(\rho_n+ \frac1{n^{4(1+\beta)/d}} \right).
\]
\end{remark}

\section{Regularity theory for model (\ref{eq-gov})}\label{sec-reg}
In this section, we recall the well-posedness result of the initial-boundary value problem \eqref{eq-gov}. For this, we make several settings. Let $L^2(\Omega)$ be the square-integrable function space with inner product $(\cdot,\cdot)_{L^2(\Omega)}$ (or $(\cdot,\cdot)$ for short) and let $H^1(\Omega)$, $H^2(\Omega)$ etc. be the usual Sobolev spaces.

The set $\{\lambda_k, \varphi_k\}_{k=1}^{\infty}$ constitutes the Dirichlet eigensystem of the elliptic operator $-\Delta: H^2(\Omega) \cap H_0^1(\Omega) \to L^2(\Omega)$, specifically,
\begin{equation*}
\begin{cases}
-\Delta \varphi_k = \lambda_k \varphi_k & \text{in } \Omega, \\
\varphi_k =0 & \text{on } \partial \Omega,
\end{cases}
\end{equation*}
where $\lambda_k$ is the eigenvalue of the operator $-\Delta$ and satisfies $0 < \lambda_1 \leq \lambda_2 \leq \ldots , \lambda_k \rightarrow \infty$ as $k \rightarrow \infty$, and $\varphi_k$ is the eigenfunctions corresponding to the value $\lambda_k$ and $\{\varphi_k\}_{k=1}^\infty$ forms an orthonormal basis in $L^2(\Omega)$. We have the asymptotic behavior of the eigenvalue $\lambda_k\sim k^{2/d}$ as $k\to\infty$. Then for $\gamma\in\mathbb R$, fractional power $(-\Delta )^{\gamma}$ can be defined 
\begin{equation*}
  (-\Delta )^{\gamma}\psi:=\sum_{k=1}^{\infty}\lambda_k^{\gamma}(\psi, \varphi _k)\varphi _k,\quad \psi \in D((-\Delta)^\gamma),
\end{equation*}
where  
$$\mathcal{D}((-\Delta )^{\gamma}):=\left \{ \psi\in L^{2}(\Omega); \sum_{k=1}^{\infty}\left |\lambda_k^{\gamma}(\psi, \varphi _k ) \right |^2<\infty  \right \}.
$$
The space $D((-\Delta)^\gamma)$ is a Hilbert space equipped with the inner product
\begin{equation*}
( \psi, \phi)_{D((-\Delta)^\gamma)}  = \left ((-\Delta)^\gamma \psi, (-\Delta)^\gamma\phi \right )_{L^2(\Omega)}.
\end{equation*}
Moreover, we define the norm 
$$
\begin{aligned}
\left \| \psi  \right \|_{\mathcal{D}((-\Delta )^\gamma) } 
&= \left ((-\Delta)^\gamma \psi, (-\Delta)^\gamma\psi \right )_{L^2(\Omega)}^{\frac12} 
= \left ( \sum_{n=1}^{\infty}\left |\lambda_n^{\gamma}(\psi, \varphi_n )  \right |^2 \right )^\frac{1}{2}.
\end{aligned}
$$
For short, we also denote the inner product $(\cdot,\cdot)_{\mathcal D((-\Delta)^\gamma)}$ and the norm $\|\cdot\|_{\mathcal D((-\Delta)^\gamma)}$ as $(\cdot,\cdot)_\gamma$ and $\|\cdot\|_\gamma$ if no conflict occurs. Furthermore, it satisfies $\mathcal{D}((-\Delta )^\gamma)\subset{H^{2\gamma}(\Omega)}$ for $\gamma>0$. In particular, we have $\mathcal{D}((-\Delta )^\frac{1}{2}) = H_0^1(\Omega)$, $\mathcal{D}((-\Delta )^{-\frac{1}{2}}) = H^{-1}(\Omega)$ and the norm equivalence $\left \| \cdot   \right \| _{\mathcal{D}((-\Delta )^\gamma  ) }\sim \left \| \cdot  \right \| _{H^{2\gamma}(\Omega)}$ with $\gamma=\pm \frac12$. 

The regularity of the solution is based on the boundedness of the Mittag-Leffler functions in Lemma \ref{lem-ml-asymp}. 
\begin{lemma}(\cite{podlubny1998fract})\label{lem-ml-asymp}
If $0<\alpha<2$, $\beta$ is an arbitrary complex number and $\mu$ is an arbitrary real number such that
$$
  \frac{\pi\alpha}2 <\mu<\min\{\pi,\pi\alpha\},
$$
then
\begin{align*}
|E_{\alpha,\beta}(z)| \le \frac{C}{1+|z|} ,\quad  \mu\le |\arg z|\le\pi,    
\end{align*}
where $E_{\alpha,\beta}(z) = \sum_{k=0}^{\infty} \frac{z^k}{\Gamma(\alpha k + \beta)}, ~ z \in \mathbb{C}$.
\end{lemma}

\begin{theorem}\label{thm-reg-maxi-1}
If $a_0=0$ and $a_1\in D((-\Delta)^{\gamma})$, $\gamma\in(\frac{d}{4},1]$, then
\[
|\partial _tu| \leq C\|(-\Delta)^{\gamma}a_1\|_{L^2}.
\]
\end{theorem}
\begin{proof}
The solution to the problem (\ref{eq-gov}) can be expressed as:
\begin{equation}\label{eq-u-wave}
  u(x,t) = \sum_{n=1}^{\infty} t E_{\alpha,2}(-\lambda_n t^\alpha)(a_1, \varphi_n) \varphi_n(x).
\end{equation}
From the definition of the Mittag-Leffler function, Lemma \ref{lem-ml-asymp} and (\ref{eq-u-wave}), one has
\begin{align*}
|\partial_tu(x,t)|
& = \bigg|\sum_{n=1}^{\infty}  E_{\alpha,1}(-\lambda_n t^\alpha)(a_1, \varphi_n) \varphi_n(x)\bigg|\\
& \le C\sum_{n=1}^{\infty}  \big|E_{\alpha,1}(-\lambda_n t^\alpha)\big|\big|(a_1, \varphi_n) \big|\\
& \le C\sum_{n=1}^{\infty} \lambda_n^{-\gamma}\big|E_{\alpha,1}(-\lambda_nt^\alpha)\big|\lambda_n^\gamma\big|(a_1, \varphi_n) \big|\\
& \le C\sqrt{\sum_{n=1}^{\infty} \lambda_n^{-2\gamma}\big|E_{\alpha,1}(-\lambda_nt^\alpha)\big|^2}\sqrt{\sum_{n=1}^{\infty}\lambda_n^{2\gamma}\big|(a_1, \varphi_n) \big|^2}\\
& \le C\|(-\Delta)^{\gamma}a_1\|_{L^2},
\end{align*}
where the last inequality follows from the inequality
\begin{align*}
\sum_{n=1}^{\infty} \lambda_n^{-2\gamma}\big|E_{\alpha,1}(-\lambda_nt^\alpha)\big|^2\le \sum_{n=1}^{\infty} \frac{\lambda_n^{-2\gamma}}{(1+\lambda_nt^\alpha)^2}\le \sum_{n=1}^{\infty} \lambda_n^{-2\gamma} \le C,\end{align*} 
where $\lambda_n\sim n^{\frac{2}{d}}$ and $\gamma\in(\frac{d}{4},1]$. This finishes the proof of the theorem.
\end{proof}

\begin{theorem}\label{thm-reg-maxi-2}
If $a_0=0$ and $a_1\in D((-\Delta)^{\gamma+\epsilon})$, $\gamma\in(\frac{d}{4},1]$ and $0<\epsilon \ll 1$, then
\[
|\partial _{tt}u| \leq Ct^{\epsilon\alpha-1}\|(-\Delta)^{\gamma+\epsilon}a_1\|_{L^2}.
\]
\end{theorem}
\begin{proof}
Similar to Theorem \ref{thm-reg-maxi-1}, one has
\begin{align*}
|\partial_{tt}u(x,t)|
& \le C\sum_{n=1}^{\infty}  \big|\lambda_n^{1-\epsilon}t^{\alpha-1}E_{\alpha,\alpha}(-\lambda_n t^\alpha)\big|\big|\lambda_n^{\epsilon}(a_1, \varphi_n) \big|\\
& \le Ct^{\epsilon\alpha-1}\sum_{n=1}^{\infty}  \big|(\lambda_nt^\alpha)^{1-\epsilon}E_{\alpha,\alpha}(-\lambda_n t^\alpha)\big|\big|\lambda_n^{\epsilon}(a_1, \varphi_n) \big|\\
& \le Ct^{\epsilon\alpha-1}\sup_n\frac{(\lambda_nt^\alpha)^{1-\epsilon}}{1+\lambda_nt^\alpha}\sum_{n=1}^{\infty}\big|\lambda_n^{\epsilon}(a_1, \varphi_n) \big|,
\end{align*}
where the last inequality follows from the inequality $\frac{(\lambda_nt^\alpha)^{1-\epsilon}}{1+\lambda_nt^\alpha}<C$, $n\geq1$.
Then, it gives that
\begin{align*}
|\partial_{tt}u(x,t)|
& \le Ct^{\epsilon\alpha-1}\sum_{n=1}^{\infty}\big|\lambda_n^{\epsilon}(a_1, \varphi_n) \big|\\
& \le Ct^{\epsilon\alpha-1}\sum_{n=1}^{\infty}\lambda_n^{-\gamma}\lambda_n^{\epsilon+\gamma}\big|(a_1, \varphi_n) \big|\\
& \le Ct^{\epsilon\alpha-1}\sqrt{\sum_{n=1}^{\infty}\lambda_n^{-2\gamma}}\sqrt{\sum_{n=1}^{\infty}\lambda_n^{2(\epsilon+\gamma)}\big|(a_1, \varphi_n) \big|^2}\\
& \le Ct^{\epsilon\alpha-1}\|(-\Delta)^{\gamma+\epsilon}a_1\|_{L^2},
\end{align*} 
where $\sqrt{\sum_{n=1}^{\infty}\lambda_n^{-2\gamma}}<C$, $\lambda_n \sim n^\frac{2}{d}$ and $\gamma\in(\frac{d}{4},1]$. We complete the proof of the theorem.
\end{proof}

\begin{lemma}\label{lem-reg-L2-1}
If $a_0=0$ and $a_1\in L^2(\Omega)$, then
\[
\|\partial_t^mu\|_{L^2}\le Ct^{1-m}\|a_1\|_{L^2},~~m=0,1.
\]
\end{lemma}
\begin{proof}
From (\ref{eq-u-wave}), indicates that
\begin{align*}
\|u\|_{L^2}^2
& = \bigg\|\sum_{n=1}^{\infty}  tE_{\alpha,2}(-\lambda_n t^\alpha)(a_1, \varphi_n) \varphi_n(x)\bigg\|_{L^2}^2\\
& = \sum_{n=1}^{\infty}  t^2E_{\alpha,2}(-\lambda_n t^\alpha)^2|(a_1, \varphi_n)|^2 \\
&\le \sup_nt^2E_{\alpha,2}(-\lambda_n t^\alpha)^2\sum_{n=1}^{\infty}|(a_1, \varphi_n)|^2\\
&\le \sup_n\frac{Ct^2}{(1+\lambda_nt^\alpha)^2}\|a_1\|_{L^2}^2
\le Ct^{2}\|a_1\|_{L^2}^2.
\end{align*}
Similarly, we get
\begin{align*}
\|\partial_tu\|_{L^2}^2
& = \bigg\|\sum_{n=1}^{\infty}  E_{\alpha,1}(-\lambda_n t^\alpha)(a_1, \varphi_n) \varphi_n(x)\bigg\|_{L^2}^2\\
& = \sum_{n=1}^{\infty} E_{\alpha,1}(-\lambda_n t^\alpha)^2|(a_1, \varphi_n)|^2 \\
&\le \sup_nE_{\alpha,1}(-\lambda_n t^\alpha)^2\sum_{n=1}^{\infty}|(a_1, \varphi_n)|^2\\
&\le \sup_n\frac{C}{(1+\lambda_nt^\alpha)^2}\|a_1\|_{L^2}^2
\le C\|a_1\|_{L^2}^2.
\end{align*}
The proof of the lemma is complete.
\end{proof}

\begin{lemma}\label{lem-reg-L2-2}
If $a_0=0$ and $a_1\in D((-\Delta)^{\gamma})$, $\gamma\in(0,1]$, then
\[
\|\partial_t^mu\|_{L^2}\le Ct^{\gamma\alpha-m+1}\|(-\Delta)^{\gamma}a_1\|_{L^2},~~m=2,3.
\]
\end{lemma}
\begin{proof}
From (\ref{eq-u-wave}), indicates that
\begin{align*}
\|\partial_{tt}u\|_{L^2}^2
& = \bigg\|\sum_{n=1}^{\infty}  \lambda_nt^{\alpha-1}E_{\alpha,\alpha}(-\lambda_n t^\alpha)(a_1, \varphi_n) \varphi_n(x)\bigg\|_{L^2}^2\\
& =\sum_{n=1}^{\infty}  \lambda_n^2t^{2\alpha-2}E_{\alpha,\alpha}(-\lambda_n t^\alpha)^2|(a_1, \varphi_n)|^2\\
& = \sum_{n=1}^{\infty}  (\lambda_nt^\alpha)^{2-2\gamma}t^{2\gamma\alpha-2}E_{\alpha,\alpha}(-\lambda_n t^\alpha)^2\lambda_n^{2\gamma}|(a_1, \varphi_n)|^2\\
& \le Ct^{2\gamma\alpha-2}\sup_n(\lambda_nt^\alpha)^{2-2\gamma}E_{\alpha,\alpha}(-\lambda_n t^\alpha)^2\sum_{n=1}^{\infty}\lambda_n^{2\gamma}|(a_1, \varphi_n)|^2\\
& \le Ct^{2\gamma\alpha-2}\sup_n\frac{(\lambda_nt^\alpha)^{2-2\gamma}}{(1+\lambda_nt^\alpha)^2}\sum_{n=1}^{\infty}\lambda_n^{2\gamma}|(a_1, \varphi_n)|^2\\
& \le Ct^{2\gamma\alpha-2}\|(-\Delta)^\gamma a_1\|_{L^2}^2,
\end{align*}
where the last inequality holds since $\sup_n\frac{(\lambda_nt^\alpha)^{2-2\gamma}}{(1+\lambda_nt^\alpha)^2}<C$.
Similarly, we get
\begin{align*}
\|\partial_{ttt}u\|_{L^2}^2
& = \bigg\|\sum_{n=1}^{\infty}  \lambda_nt^{\alpha-2}E_{\alpha,\alpha-1}(-\lambda_n t^\alpha)(a_1, \varphi_n) \varphi_n(x)\bigg\|_{L^2}^2\\
& = \sum_{n=1}^{\infty}  (\lambda_nt^\alpha)^{2-2\gamma}t^{2\gamma\alpha-4}E_{\alpha,\alpha-1}(-\lambda_n t^\alpha)^2\lambda_n^{2\gamma}|(a_1, \varphi_n)|^2\\
& \le Ct^{2\gamma\alpha-4}\sup_n(\lambda_nt^\alpha)^{2-2\gamma}E_{\alpha,\alpha-1}(-\lambda_n t^\alpha)^2\sum_{n=1}^{\infty}\lambda_n^{2\gamma}|(a_1, \varphi_n)|^2\\
& \le Ct^{2\gamma\alpha-4}\|(-\Delta)^\gamma a_1\|_{L^2}^2.
\end{align*}
Collecting all the above estimates, we finish the proof of the lemma.
\end{proof}

\begin{theorem}\label{thm-reg-delta-1}
If $a_0=0$ and $a_1\in D((-\Delta)^{\gamma+\epsilon})$, $\gamma\in \mathbb R$ and $\epsilon \in(0,1]$, then
\[
\|\partial_t^m(-\Delta)^{\gamma+\epsilon}u\|_{L^2}\le Ct^{1-m}\|(-\Delta)^{\gamma+\epsilon}a_1\|_{L^2},~~m=0,1,
\]
\[
\|\partial_t^m(-\Delta)^{\gamma}u\|_{L^2}\le Ct^{\epsilon\alpha-m+1}\|(-\Delta)^{\gamma+\epsilon}a_1\|_{L^2},~~m=2,3.
\]
\end{theorem}
\begin{proof} The desired results are directly obtained following the idea in lemmas \ref{lem-reg-L2-1} and \ref{lem-reg-L2-2}.
\end{proof}

\begin{theorem}\label{thm-reg-delta-2}
If $a_0=0$ and $a_1\in D((-\Delta)^{\gamma+\epsilon})$, $\gamma\in(\frac14,1]$ and $0<\epsilon \ll 1$, then
\[
\|\partial_t^m(-\Delta)^{\frac12}u\|_{L^2}\le Ct^{1-m-\frac{\alpha}{4}}\|(-\Delta)^{\gamma+\epsilon}a_1\|_{L^2},~~m=0,1,2,3.
\]
\end{theorem}
\begin{proof} 
For $\xi\in(0,1]$, it gives that
\begin{align*}
\|(-\Delta)^{\frac12}u\|_{L^2}^2
& = \bigg\|\sum_{n=1}^{\infty}  \lambda_n^{\frac12}tE_{\alpha,2}(-\lambda_n t^\alpha)(a_1, \varphi_n) \varphi_n(x)\bigg\|_{L^2}^2\\
&=\sum_{n=1}^\infty\lambda_nt^2E_{\alpha,2}(-\lambda_n t^\alpha)^2|(a_1, \varphi_n)|^2\\
&=\sum_{n=1}^\infty t^{\alpha(2\xi-2)+2}(\lambda_nt^\alpha)^{2-2\xi}E_{\alpha,2}(-\lambda_n t^\alpha)^2\lambda_n^{2\xi-1}|(a_1, \varphi_n)|^2\\
&\le\sup_n\frac{(\lambda_nt^\alpha)^{2-2\xi}}{(1+\lambda_nt^\alpha)^2}t^{\alpha(2\xi-2)+2}\sum_{n=1}^\infty \lambda_n^{2\xi-1}|(a_1, \varphi_n)|^2\\
&\le Ct^{\alpha(2\xi-2)+2}\|(-\Delta)^{\xi-\frac12}a_1\|_{L^2}^2,
\end{align*}
where the last inequality holds since $\sup_n\frac{(\lambda_nt^\alpha)^{2-2\xi}}{(1+\lambda_nt^\alpha)^2}<C$.
We take $\xi=\frac34\leq\gamma+\epsilon+\frac12$, it arrives that
\begin{align*}
\|(-\Delta)^{\frac12}u\|_{L^2}
&\le Ct^{1-\frac{\alpha}{4}}\|(-\Delta)^{\gamma+\epsilon}a_1\|_{L^2}.
\end{align*}
Similarly, one has $\|\partial_t(-\Delta)^{\frac12}u\|_{L^2}\leq Ct^{-\frac{\alpha}{4}}\|(-\Delta)^{\gamma+\epsilon}a_1\|_{L^2}$. The proof is complete.
\end{proof}

Some properties of $E_{\alpha,\beta}(-\lambda_nt^\alpha)$, $\alpha>0$, $\beta\in\mathbb{R}$, are need to deduce the estimates of $\|\partial_t^m(-\Delta)^{\frac12}u\|_{L^2}$, $m=2,3$, 
\begin{align*}
\partial_t[tE_{\alpha,1}(-\lambda_n t^\alpha)]
&= \partial_t\bigg[t\sum_{k=0}^\infty\frac{(-\lambda_nt^\alpha)^k}{\Gamma{(k\alpha+1)}}\bigg]\\
&= \partial_t\bigg[t+\sum_{k=1}^\infty\frac{(-\lambda_n)^kt^{k\alpha+1}}{\Gamma{(k\alpha+1)}}\bigg]\\
&= 1+\sum_{k=1}^\infty\frac{(-\lambda_nt^{\alpha})^k}{\Gamma{(k\alpha+1)}}+\sum_{k=1}^\infty\frac{(-\lambda_nt^{\alpha})^k}{\Gamma{(k\alpha)}}\\
&:=E_{\alpha,1}(-\lambda_nt^\alpha)+E_{\alpha,0}(-\lambda_nt^\alpha),
\end{align*}
and, we get
\begin{align*}
\partial_t[E_{\alpha,1}(-\lambda_n t^\alpha)]
&=\partial_t[t^{-1}tE_{\alpha,1}(-\lambda_n t^\alpha)]\\
&=-t^{-1}E_{\alpha,1}(-\lambda_n t^\alpha)+t^{-1}\partial_t[tE_{\alpha,1}(-\lambda_n t^\alpha)]\\
&=-t^{-1}E_{\alpha,1}(-\lambda_n t^\alpha)+t^{-1}\big(E_{\alpha,1}(-\lambda_nt^\alpha)+E_{\alpha,0}(-\lambda_nt^\alpha)\big)\\
&=t^{-1}E_{\alpha,0}(-\lambda_nt^\alpha).
\end{align*}
Based on above results, it holds that
\begin{align*}
\|\partial_t^2(-\Delta)^{\frac12}u\|_{L^2}^2
& = \bigg\|\sum_{n=1}^{\infty}  \lambda_n^{\frac12}t^{-1}E_{\alpha,0}(-\lambda_n t^\alpha)(a_1, \varphi_n) \varphi_n(x)\bigg\|_{L^2}^2\\
&=\sum_{n=1}^{\infty}  \lambda_nt^{-2}E_{\alpha,0}(-\lambda_n t^\alpha)^2|(a_1, \varphi_n)|^2\\
&=\sum_{n=1}^\infty t^{\alpha(2\xi-2)-2}(\lambda_nt^\alpha)^{2-2\xi}E_{\alpha,0}(-\lambda_n t^\alpha)^2\lambda_n^{2\xi-1}|(a_1, \varphi_n)|^2\\
&\le\sup_n\frac{(\lambda_nt^\alpha)^{2-2\xi}}{(1+\lambda_nt^\alpha)^2}t^{\alpha(2\xi-2)-2}\sum_{n=1}^\infty \lambda_n^{2\xi-1}|(a_1, \varphi_n)|^2\\
&\le Ct^{\alpha(2\xi-2)-2}\|(-\Delta)^{\xi-\frac12}a_1\|_{L^2}^2,
\end{align*}
where $\xi\in(0,1]$, taking $\xi=\frac34\leq\gamma+\epsilon+\frac12$, i. e.
\begin{align*}
\|\partial_t^2(-\Delta)^{\frac12}u\|_{L^2}
&\le Ct^{-\frac{\alpha}{4}-1}\|(-\Delta)^{\gamma+\epsilon}a_1\|_{L^2}^2.
\end{align*}
Similarly, we get
\begin{align*}
\|\partial_t^3(-\Delta)^{\frac12}u\|_{L^2}
&\le Ct^{-\frac{\alpha}{4}-2}\|(-\Delta)^{\gamma+\epsilon}a_1\|_{L^2}^2.
\end{align*}

\begin{theorem}\label{thm-reg-L2-v}
If $a_0=0$ and $a_1\in L^2(\Omega)$, then
\[
\|\partial_t^mv\|_{L^2}\le Ct^{1-\nu-m}\|a_1\|_{L^2},~~m=0,1,2,3.
\]
\end{theorem}
\begin{proof}
From (\ref{eq-u-wave}), indicates that
\begin{align*}
\|v\|_{L^2}^2
& = \bigg\|\frac{1}{\Gamma{(1-\nu)}}\int_0^t(t-s)^{-\nu}\partial_su(s)ds\bigg\|_{L^2}^2\\
& = \bigg\|\frac{1}{\Gamma{(1-\nu)}}\sum_{n=1}^{\infty}  \int_0^t(t-s)^{-\nu}E_{\alpha,1}(-\lambda_n s^\alpha)ds(a_1, \varphi_n) \varphi_n(x)\bigg\|_{L^2}^2\\
& = \sum_{n=1}^{\infty}\big(t^{1-\nu}E_{\alpha,2-\nu}(-\lambda_n t^\alpha)\big)^2|(a_1, \varphi_n)|^2 \\
&\le t^{2-2\nu}\sup_nE_{\alpha,2-\nu}(-\lambda_n t^\alpha)^2\sum_{n=1}^{\infty}|(a_1, \varphi_n)|^2\\
&\le Ct^{2-2\nu}\|a_1\|_{L^2}^2.
\end{align*}
Similarly, we get
\begin{align*}
\|\partial_t^mv\|_{L^2}^2
& = \sum_{n=1}^{\infty}\big[\partial_t^m\big(t^{1-\nu}E_{\alpha,2-\nu}(-\lambda_n t^\alpha)\big)\big]^2|(a_1, \varphi_n)|^2\\
& = t^{2-2\nu-2m}\sum_{n=1}^{\infty}\big(E_{\alpha,2-\nu-m}(-\lambda_n t^\alpha)\big)^2|(a_1, \varphi_n)|^2\\
&\le Ct^{2-2\nu-2m}\|a_1\|_{L^2}^2.
\end{align*}
We complete the proof of the theorem.
\end{proof}

\section{Numerical experiment}\label{sec-num}
In this section, we carry out some numerical experiments to check the theoretical results. 
In Example \ref{ex1}, we consider testing the convergence statements for one dimensional cases. 
Let us return to the literature where the SFOR method is proposed \cite{LyuP2022SFOR}. Remark 2.2 therein explains why auxiliary variables were introduced to extract the singular term $a_1(x)\omega_{2-\alpha}(t)$. Based on this, they adopted the following numerical framework:
\begin{equation}\label{eq-gov-ref}
\begin{cases}
  \partial_t^\nu \textbf{v}  - \Delta \textbf{u} = t\Delta a_1(x), & (x,t) \in \Omega \times (0,T), \\
  \textbf{v}  = \partial_t^\nu \textbf{u}, & (x,t) \in \Omega \times (0,T), \\
  \textbf{u}(x,0) = \textbf{v}(x,0)=0, & x \in \Omega, \\
  \textbf{u}(x,t) = \textbf{v}(x,t)=0, & (x,t) \in \partial \Omega \times (0,T),
\end{cases}
\end{equation}
where $u=\textbf{u}+ta_1(x)$. The limitation of (\ref{eq-gov-ref}) is $(\Delta a_1(x),\phi(x))$, $x\in\Omega$ exists, where $\phi(x)$ is basis function from finite element space.
Our numerical framework (\ref{eq-gov-trans}) relaxes this requirement. Furthermore, we find that the optimal convergence is reached, despite the presence of $a_1(x)\omega_{2-\alpha}(t)$, when the mesh parameter $r=\frac{4-\alpha}{2-\alpha}$, $1<\alpha<2-\epsilon$, $\epsilon$ is a fixed positive constant. 

Theoretically, it could not work for the case $\alpha\rightarrow2^-$ since $r=\frac{4-\alpha}{2-\alpha}\rightarrow +\infty$. For  $\Delta a_1\in L^2(\Omega)$, scheme (\ref{eq-gov-ref}) is helpful. The reason is presented in Remark \ref{ill-ref}. For more general $a_1(x)$, we can use scheme (\ref{eq-gov-trans}) with bounded $r$. It gives detailed instructions on how to apply the SFOR method in different application cases.

\begin{remark}\label{ill-ref}
The point is \textbf{v} become more regular. Following the idea of Theorem \ref{thm-reg-L2-v}, one has 
$\|\partial_t^m\textbf{v}\|_{L^2}\leq Ct^{1+\nu-m}\|a_1\|_{L^2}$, $m=0,1,2,3$. Combining the analysis in Theorems \ref{thm-conv-1} and \ref{thm-conv-2}, it implies that the mesh parameter $r=2$ is enough.
\end{remark}

\begin{example}(One dimensional)\label{ex1}
Forward problems (\ref{eq-gov}) with $\Omega=(0,\pi)$, $T=0.1$,
\begin{align*}
(a) ~~a_1(x)&=\sin(x),\\
(b)~~a_1(x)&=x,~~x\in(0,\pi/2],~~a_1(x)=\pi-x,~~x\in(\pi/2,\pi).
\end{align*}
The size of the space grids $h=\frac{\pi}{200}$, $N$ is the number of partitions in the time
grids. $e_{L^2}=\max_{1\leq n\leq N}\|u_h^n-U_h^n\|_{L^2}$, where $u_h^n$ and $U_h^n$ are the reference solution ($h=\frac{\pi}{200},N=2048$) and the numerical solution, respectively. Furthermore, to test the convergence rate, let $Order=\log_2(e_{L^2}(N/2)/e_{L^2}(N))$. 
\end{example}

\begin{example}(One dimensional)\label{ex2}
Backward problems (\ref{eq-gov}) with $\Omega=(0,\pi)$, $T=0.1$,
\begin{align*}
(a) ~~a_1(x)&=\sin(x),~~\alpha=1.5,\\
(b)~~a_1(x)&=x,~~x\in(0,\pi/2],~~a_1(x)=\pi-x,~~x\in(\pi/2,\pi),~~\alpha=1.5.
\end{align*}
We use the numerical framework (\ref{num-scheme}) with the mesh parameter $r=\frac{4-\alpha}{2-\alpha}$, where the size of the space grids $h=\frac{\pi}{20}$, $N=2048$ is the number of partitions in the time grids. Reconstruction $a_1^{rec}$ is found by gradient descent method with initial guess $a_1^{rec}=0$.
\end{example}

\begin{table}[!ht]
\caption{Scheme (\ref{eq-gov-trans}) for Example \ref{ex1} (a) with $\alpha=1.25$.}
\renewcommand{\arraystretch}{1}
\def\temptablewidth{1\textwidth}
\begin{center}
 \begin{tabular*}{\temptablewidth}{@{\extracolsep{\fill}}lcccccc}\hline
 $N$ & $r=1$   &   &$r_{opt}=\frac{4-\alpha}{2-\alpha}$      &     &$r=\frac{4-\alpha}{\alpha}$  &    \\ 
 \cline{2-3}\cline{4-5}\cline{6-7}
 & $e_{L^{2}}(N)$         &Order     &$e_{L^{2}}(N)$    &Order    &$e_{L^{2}}(N)$ &Order     \\ \hline
 16 &9.7511e-03     &-      &3.2933e-03   &-       &3.5919e-03  &-  \\ 
 32 &5.8969e-03     &0.726  &1.3099e-03   &1.330   &1.5200e-03  &1.241    \\ 
 64 &3.4701e-03     &0.765  &5.0947e-04   &1.362   &6.2187e-04  &1.289 \\ 
 128 &1.9768e-03    &0.812  &1.9505e-04   &1.385   &2.4760e-04  &1.329    \\ 
 Optimal order  &\multicolumn{6}{c}{1.375} \\ \hline
\end{tabular*}
\end{center}
\end{table}
In example \ref{ex2}, we consider recovering the initial function $a_1$. For the case $(a)$, let $\sigma=0.05$, then the noise lever $\frac{\sigma}{\|Sa_1\|_{L^\infty}}\approx 50\%$. The number of observation points $n$ are taken $11$, $49$ and $199$ in numerical tests. The optimal regularization parameter of $H^1$ regularization $\rho_n = O((\sigma n^{-\frac12} \|a^*\|^{-1}_{H^1})^{12/7})=O(5.1\times10^{-4})$, $O(1.4\times10^{-4})$ and $O(4.3\times10^{-5})$. In Figure \ref{figure1}, 11 observation points are presented. Then, $L^2$ error $\frac{\|a_1^{rec}-a_1\|_{L^2}}{\|a_1\|_{L^2}}$ is given when $\rho_{11}=10^{-k}$. Furthermore, we compare $a_1$ with the optimal $a_1^{rec}$ when $\rho_{11}=10^{-3.25}\approx 5.6\times10^{-4}$. It shows that there is a big gap between reconstruction $a_1^{rec}$ and $a_1$. To obtain a better $a_1^{rec}$, increasing observational data is a viable strategy. The related results in Figures \ref{figure4}-\ref{figure9} support the conclusion of Remark \ref{remark-1}.

For the case $(b)$, let $\sigma=0.015$, then the noise lever $\frac{\sigma}{\|Sa_1\|_{L^\infty}}\approx 10\%$. The number of observation points $n$ is taken $199$ in numerical tests. The optimal regularization parameter $\rho_{199} = O((\sigma n^{-\frac12} \|a^*\|^{-1}_{H^1})^{12/7})=O(3.0\times10^{-6})$. $a_1^{rec}$ found by $H^1$ regularization with $\rho_{199}=3\times 10^{-6}$ is presented in Figure \ref{figure11}. It verifies our method also work for $a_1$ with singularity point.

\begin{example}(Two dimensional)\label{ex3}
Forward problems (\ref{eq-gov}) with $\Omega=(0,1)\times(0,1)$, $T=0.1$,
\begin{align*}
a_1(x,y)&=\ln(1+10x)(x-1)\sin^{\frac{3}{4}}(\pi y).
\end{align*}
The size of the space grids $h=\frac{1}{30}$, $N$ is the number of partitions in the time
grids. Let the numerical solution on fine mesh be the reference solution ($h=\frac{1}{30}$,$N=1280$). 
\end{example}
In Example \ref{ex3}, when we take $r=\frac{4-\alpha}{2-\alpha}$, the convergence rate reaches optimal $2-\alpha/2$.
Under the same accuracy requirements, the use of optimal mesh parameters can save the cost of calculation. Furthermore, it improves the efficiency of numerical inversion.

\begin{example}(Two dimensional)\label{ex4}
Backward problems (\ref{eq-gov}) with $\Omega=(0,1)\times(0,1)$, $T=0.1$,
\begin{align*}
a_1(x,y)&=\ln(1+10x)(x-1)\sin^{\frac{3}{4}}(\pi y),~~\alpha=1.25.
\end{align*}
We use the numerical framework (\ref{num-scheme}) with the mesh parameter $r=\frac{4-\alpha}{2-\alpha}$, where the size of the space grids $h=\frac{1}{30}$, $N=160$ is the number of partitions in the time grids. Reconstruction $a_1^{rec}$ is found by gradient descent method with initial guess $a_1^{rec}=0$.
\end{example}

In Example \ref{ex4}, let $\sigma=0.01$, then the noise lever $\frac{\sigma}{\|Sa_1\|_{L^\infty}}\approx 12\%$.
We take the points of space grids as measurement points. The optimal regularization parameter of $H^1$ regularization $\rho_{841} = O((\sigma n^{-\frac12} \|a^*\|^{-1}_{H^1})^{3/2})=O(1.5\times10^{-6})$. The reconstruction results are presented with different regularization parameters.  $L^2$ errors $\frac{\|a_1^{rec}-a_1\|_{L^2}}{\|a_1\|_{L^2}}=15.0\%$, $11.0\%$ and $24.5\%$ when $\rho_{841}=10^{-5}$, $2\times10^{-6}$ and $10^{-7}$, respectively.
$a_1^{rec}$ with $\rho_{841}=2\times10^{-6}$ close to the exact one, see Figures \ref{figure12} and \ref{figure14}.
However, the peak in Figure \ref{figure13} is not consistent with that of the exact one. And in Figure \ref{figure15}, the reconstruction becomes blurry when $\rho_{841}$ is too small.

\begin{table}[!ht]
\caption{Scheme (\ref{eq-gov-trans}) for Example \ref{ex1} (a) with $\alpha=1.5$.}
\renewcommand{\arraystretch}{1}
\def\temptablewidth{1\textwidth}
\begin{center}
 \begin{tabular*}{\temptablewidth}{@{\extracolsep{\fill}}lcccccc}\hline
 $N$ & $r=1$   &   &$r_{opt}=\frac{4-\alpha}{2-\alpha}$      &     &$r=\frac{4-\alpha}{\alpha}$  &    \\ 
 \cline{2-3}\cline{4-5}\cline{6-7}
 & $e_{L^{2}}(N)$         &Order     &$e_{L^{2}}(N)$    &Order    &$e_{L^{2}}(N)$ &Order     \\ \hline
 16 &2.4091e-02     &-      &8.1512e-03   &-       &1.4993e-02  &-  \\ 
 32 &1.6884e-02     &0.513  &3.5616e-03   &1.195   &8.8464e-03  &0.761    \\ 
 64 &1.1451e-02     &0.560  &1.5057e-03   &1.242   &5.0422e-03  &0.811 \\ 
 128 &7.4624e-03    &0.618  &6.2368e-04   &1.272   &2.7729e-03  &0.863    \\ 
 Optimal order  &\multicolumn{6}{c}{1.25} \\ \hline
\end{tabular*}
\end{center}
\end{table}

\begin{table}[!ht]
\caption{Scheme (\ref{eq-gov-trans}) for Example \ref{ex1} (a) with $\alpha=1.75$.}
\renewcommand{\arraystretch}{1}
\def\temptablewidth{1\textwidth}
\begin{center}
 \begin{tabular*}{\temptablewidth}{@{\extracolsep{\fill}}lcccccc}\hline
 $N$ & $r=1$   &   &$r_{opt}=\frac{4-\alpha}{2-\alpha}$      &     &$r=\frac{4-\alpha}{\alpha}$  &    \\ 
 \cline{2-3}\cline{4-5}\cline{6-7}
 & $e_{L^{2}}(N)$         &Order     &$e_{L^{2}}(N)$    &Order    &$e_{L^{2}}(N)$ &Order     \\ \hline
 16 &4.2309e-02     &-      &2.3254e-02   &-       &4.1819e-02  &-  \\ 
 32 &3.3563e-02     &0.334  &1.1796e-02   &0.979   &3.2089e-02  &0.382    \\ 
 64 &2.5699e-02     &0.385  &5.7768e-03   &1.030   &2.3768e-02  &0.433 \\ 
 128 &1.8823e-02    &0.449  &2.7613e-03   &1.065   &1.6855e-02  &0.496    \\ 
 Optimal order  &\multicolumn{6}{c}{1.125} \\ \hline
\end{tabular*}
\end{center}
\end{table}

\begin{table}[!ht]
\caption{Scheme (\ref{eq-gov-ref}) for Example \ref{ex1} (a).}
\renewcommand{\arraystretch}{1}
\def\temptablewidth{1\textwidth}
\begin{center}
 \begin{tabular*}{\temptablewidth}{@{\extracolsep{\fill}}lcccccc}\hline
 $N$ & $\alpha=1.01$   &$r_{opt}=2$   &$\alpha=1.01$      &$r=\frac{4-\alpha}{\alpha}$     &$\alpha=1.99$  &$r_{opt}=2$    \\ 
 \cline{2-3}\cline{4-5}\cline{6-7}
 & $e_{L^{2}}(N)$         &Order     &$e_{L^{2}}(N)$    &Order    &$e_{L^{2}}(N)$ &Order     \\ \hline
 16 &1.2195e-04     &-      &1.8560e-04   &-       &5.5727e-05  &-  \\ 
 32 &4.5553e-05     &1.421  &7.0649e-05   &1.394   &2.6807e-05  &1.056    \\ 
 64 &1.6669e-05     &1.450  &2.6175e-05   &1.433   &1.2964e-05  &1.048 \\ 
 128 &5.9895e-06    &1.477  &9.4821e-06   &1.465   &6.2090e-06  &1.062    \\ 
 \cline{2-5}\cline{6-7}
 Optimal order  &\multicolumn{4}{c}{1.495} &\multicolumn{2}{c}{1.005} \\ \hline
\end{tabular*}
\end{center}
\end{table}

\begin{table}[!ht]
\caption{Scheme (\ref{eq-gov-trans}) for Example \ref{ex1} (b) with $\alpha=1.25$.}
\renewcommand{\arraystretch}{1}
\def\temptablewidth{1\textwidth}
\begin{center}
 \begin{tabular*}{\temptablewidth}{@{\extracolsep{\fill}}lcccccc}\hline
 $N$ & $r=1$   &   &$r_{opt}=\frac{4-\alpha}{2-\alpha}$      &     &$r=\frac{4-\alpha}{\alpha}$  &    \\ 
 \cline{2-3}\cline{4-5}\cline{6-7}
 & $e_{L^{2}}(N)$         &Order     &$e_{L^{2}}(N)$    &Order    &$e_{L^{2}}(N)$ &Order     \\ \hline
 16 &1.2467e-02     &-      &4.2304e-03   &-       &4.6017e-03  &-  \\ 
 32 &7.5376e-03     &0.726  &1.6830e-03   &1.330   &1.9470e-03  &1.241    \\ 
 64 &4.4349e-03     &0.765  &6.5472e-04   &1.362   &7.9643e-04  &1.290 \\ 
 128 &2.5261e-03    &0.812  &2.5069e-04   &1.385   &3.1707e-04  &1.328    \\ 
 Optimal order  &\multicolumn{6}{c}{1.375} \\ \hline
\end{tabular*}
\end{center}
\end{table}

\begin{table}[!ht]
\caption{Scheme (\ref{eq-gov-trans}) for Example \ref{ex1} (b) with $\alpha=1.5$.}
\renewcommand{\arraystretch}{1}
\def\temptablewidth{1\textwidth}
\begin{center}
 \begin{tabular*}{\temptablewidth}{@{\extracolsep{\fill}}lcccccc}\hline
 $N$ & $r=1$   &   &$r_{opt}=\frac{4-\alpha}{2-\alpha}$      &     &$r=\frac{4-\alpha}{\alpha}$  &    \\ 
 \cline{2-3}\cline{4-5}\cline{6-7}
 & $e_{L^{2}}(N)$         &Order     &$e_{L^{2}}(N)$    &Order    &$e_{L^{2}}(N)$ &Order     \\ \hline
 16 &3.0851e-02     &-      &1.0472e-02   &-       &1.9206e-02  &-  \\ 
 32 &2.1619e-02     &0.513  &4.5766e-03   &1.194   &1.1330e-02  &0.761    \\ 
 64 &1.4660e-02     &0.560  &1.9351e-03   &1.242   &6.4568e-03  &0.811 \\ 
 128 &9.5536e-03    &0.618  &8.0164e-04   &1.271   &3.5505e-03  &0.863    \\ 
 Optimal order  &\multicolumn{6}{c}{1.25} \\ \hline
\end{tabular*}
\end{center}
\end{table}

\begin{table}[!ht]
\caption{Scheme (\ref{eq-gov-trans}) for Example \ref{ex1} (b) with $\alpha=1.75$.}
\renewcommand{\arraystretch}{1}
\def\temptablewidth{1\textwidth}
\begin{center}
 \begin{tabular*}{\temptablewidth}{@{\extracolsep{\fill}}lcccccc}\hline
 $N$ & $r=1$   &   &$r_{opt}=\frac{4-\alpha}{2-\alpha}$      &     &$r=\frac{4-\alpha}{\alpha}$  &    \\ 
 \cline{2-3}\cline{4-5}\cline{6-7}
 & $e_{L^{2}}(N)$         &Order     &$e_{L^{2}}(N)$    &Order    &$e_{L^{2}}(N)$ &Order     \\ \hline
 16 &5.4227e-02     &-      &2.9853e-02   &-       &5.3599e-02  &-  \\ 
 32 &4.3015e-02     &0.334  &1.5145e-02   &0.979   &4.1127e-02  &0.382    \\ 
 64 &3.2936e-02     &0.385  &7.4165e-03   &1.030   &3.0461e-02  &0.433 \\ 
 128 &2.4123e-02    &0.449  &3.5449e-03   &1.065   &2.1601e-02  &0.496    \\ 
 Optimal order  &\multicolumn{6}{c}{1.125} \\ \hline
\end{tabular*}
\end{center}
\end{table}

\begin{table}[!ht]
\caption{Scheme (\ref{eq-gov-trans}) for Example \ref{ex1} (b) with $\alpha=1.99$.}
\renewcommand{\arraystretch}{1}
\def\temptablewidth{1\textwidth}
\begin{center}
 \begin{tabular*}{\temptablewidth}{@{\extracolsep{\fill}}lcccccc}\hline
 $N$ & $r=1$   &   &$r=5$      &     &$r=10$  &    \\ 
 \cline{2-3}\cline{4-5}\cline{6-7}
 & $e_{L^{2}}(N)$         &Order     &$e_{L^{2}}(N)$    &Order    &$e_{L^{2}}(N)$ &Order     \\ \hline
 16 &7.0532e-03     &-      &3.0649e-02   &-       &5.3655e-02  &-  \\ 
 32 &6.1265e-03     &0.203  &2.6342e-02   &0.218   &4.6214e-02  &0.215    \\ 
 64 &5.1452e-03     &0.252  &2.1830e-02   &0.271   &3.8240e-02  &0.273 \\ 
 128 &4.1324e-03    &0.316  &1.7280e-02   &0.337   &3.0184e-02  &0.341    \\ 
 Optimal order  &\multicolumn{6}{c}{1.005} \\ \hline
\end{tabular*}
\end{center}
\end{table}

\begin{figure}[!htb]
        \begin{minipage}{0.33\linewidth}
		\centering
		\setlength{\abovecaptionskip}{0.28cm}
		\includegraphics[width=1\linewidth]{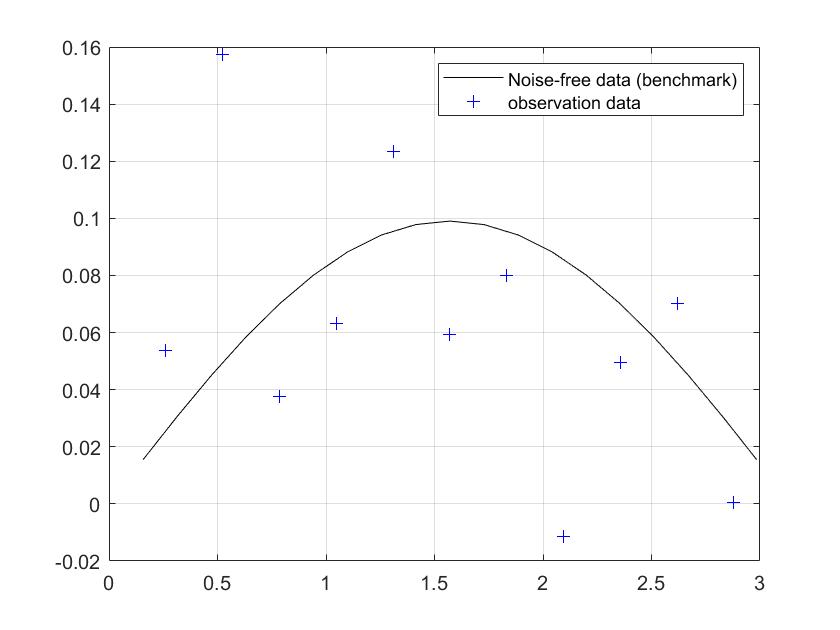}
		\caption{Observation.}
		\label{figure1}
	\end{minipage}
        \begin{minipage}{0.33\linewidth}
		\centering
		\setlength{\abovecaptionskip}{0.28cm}
		\includegraphics[width=1\linewidth]{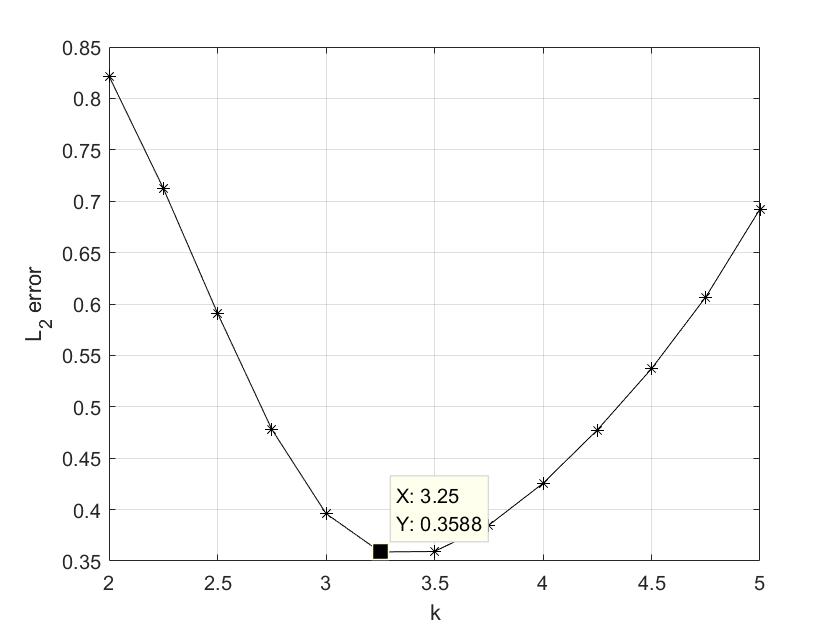}
		\caption{$L^2$ error.}
	  \end{minipage}
        \begin{minipage}{0.33\linewidth}
		\centering
		\setlength{\abovecaptionskip}{0.28cm}
		\includegraphics[width=1\linewidth]{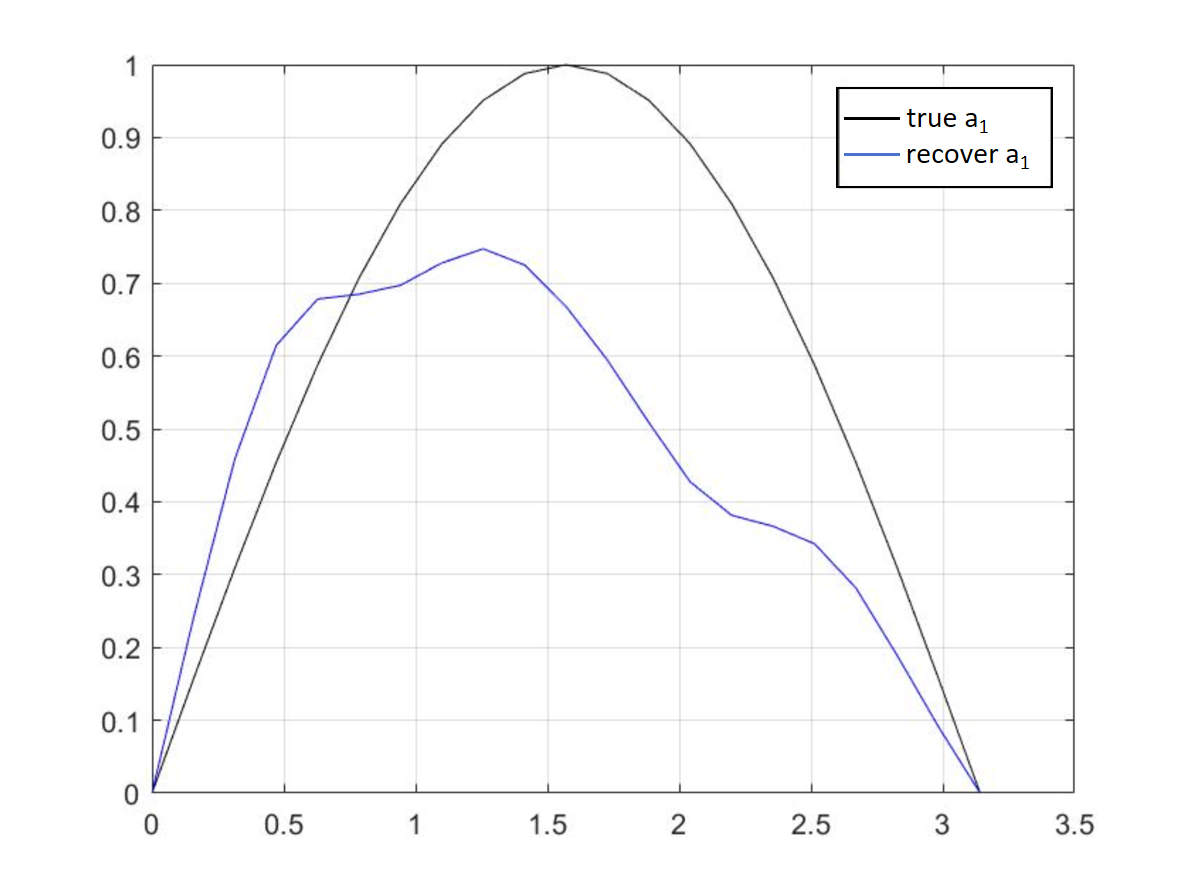}
		\caption{$\rho_{11}=10^{-3.25}$.}
	  \end{minipage}
    
\end{figure}

\begin{figure}[!htb]
        \begin{minipage}{0.33\linewidth}
		\centering
		\setlength{\abovecaptionskip}{0.28cm}
		\includegraphics[width=1\linewidth]{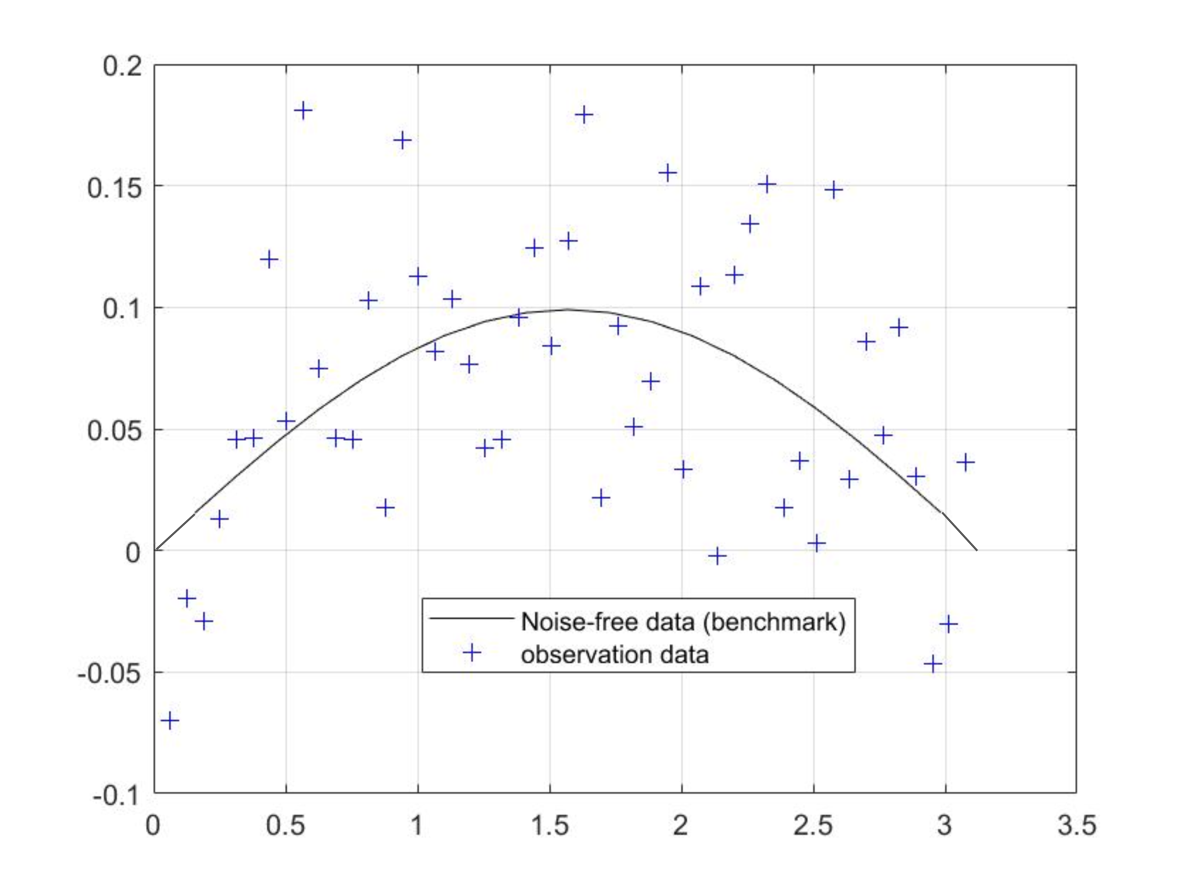}
		\caption{Observation.}
		\label{figure4}
	\end{minipage}
        \begin{minipage}{0.33\linewidth}
		\centering
		\setlength{\abovecaptionskip}{0.28cm}
		\includegraphics[width=1\linewidth]{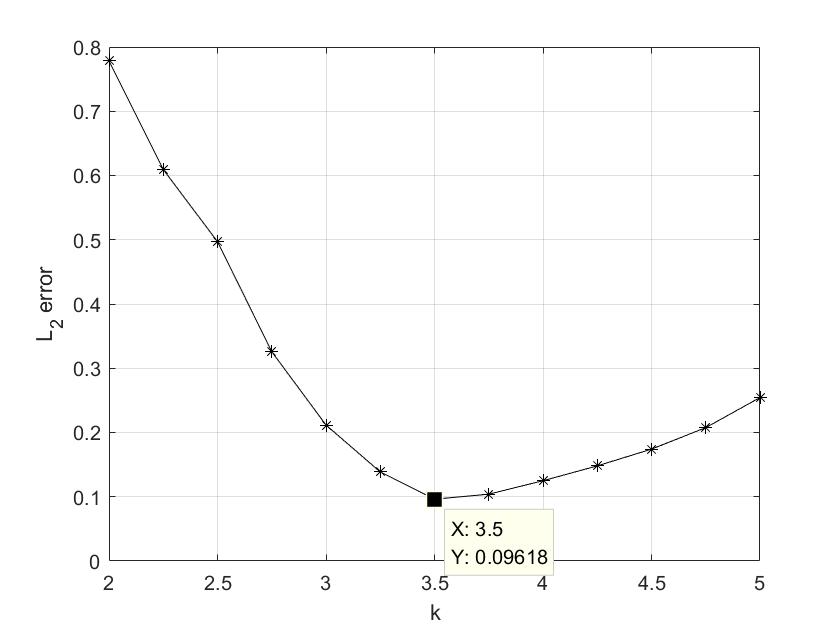}
		\caption{$L^2$ error.}
	  \end{minipage}
        \begin{minipage}{0.33\linewidth}
		\centering
		\setlength{\abovecaptionskip}{0.28cm}
		\includegraphics[width=1\linewidth]{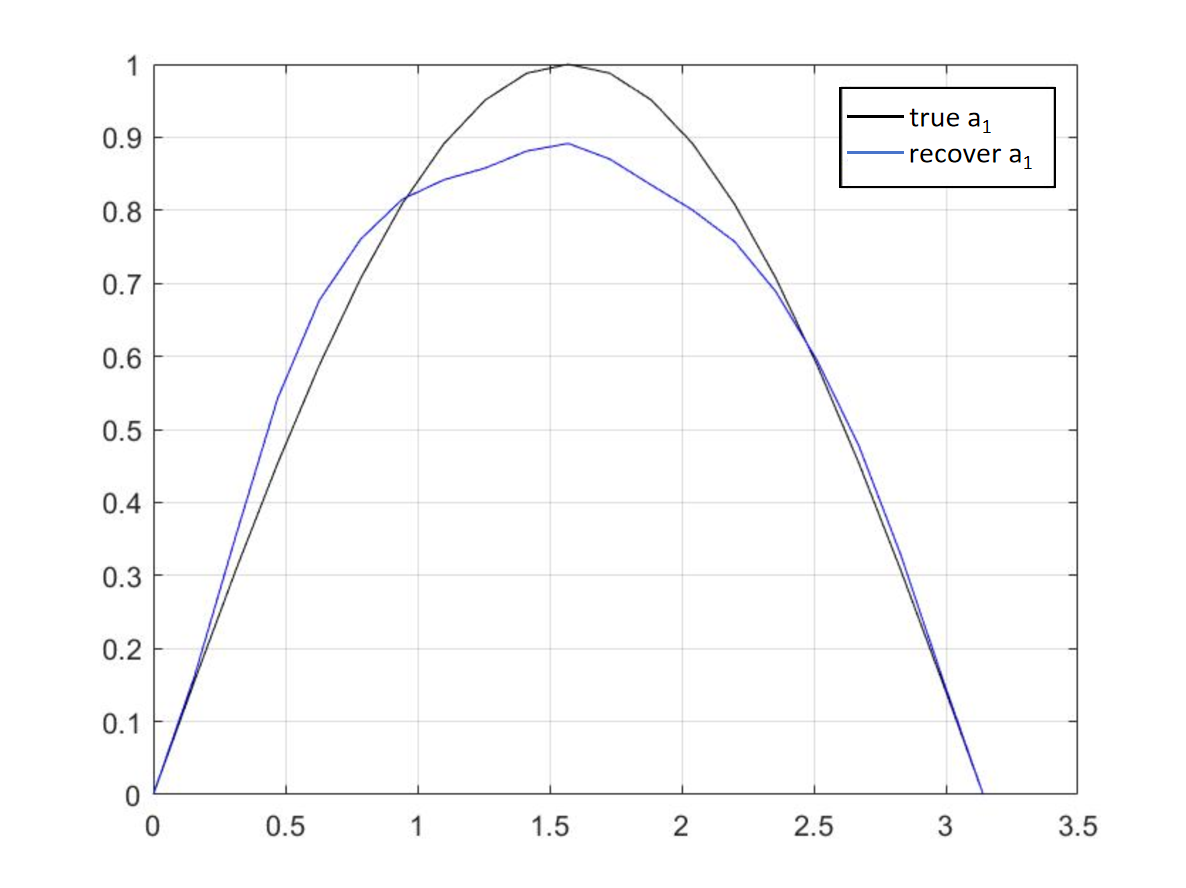}
		\caption{$\rho_{49}=10^{-3.5}$.}
	  \end{minipage}
    
\end{figure}

\begin{figure}[!htb]
        \begin{minipage}{0.33\linewidth}
		\centering
		\setlength{\abovecaptionskip}{0.28cm}
		\includegraphics[width=1\linewidth]{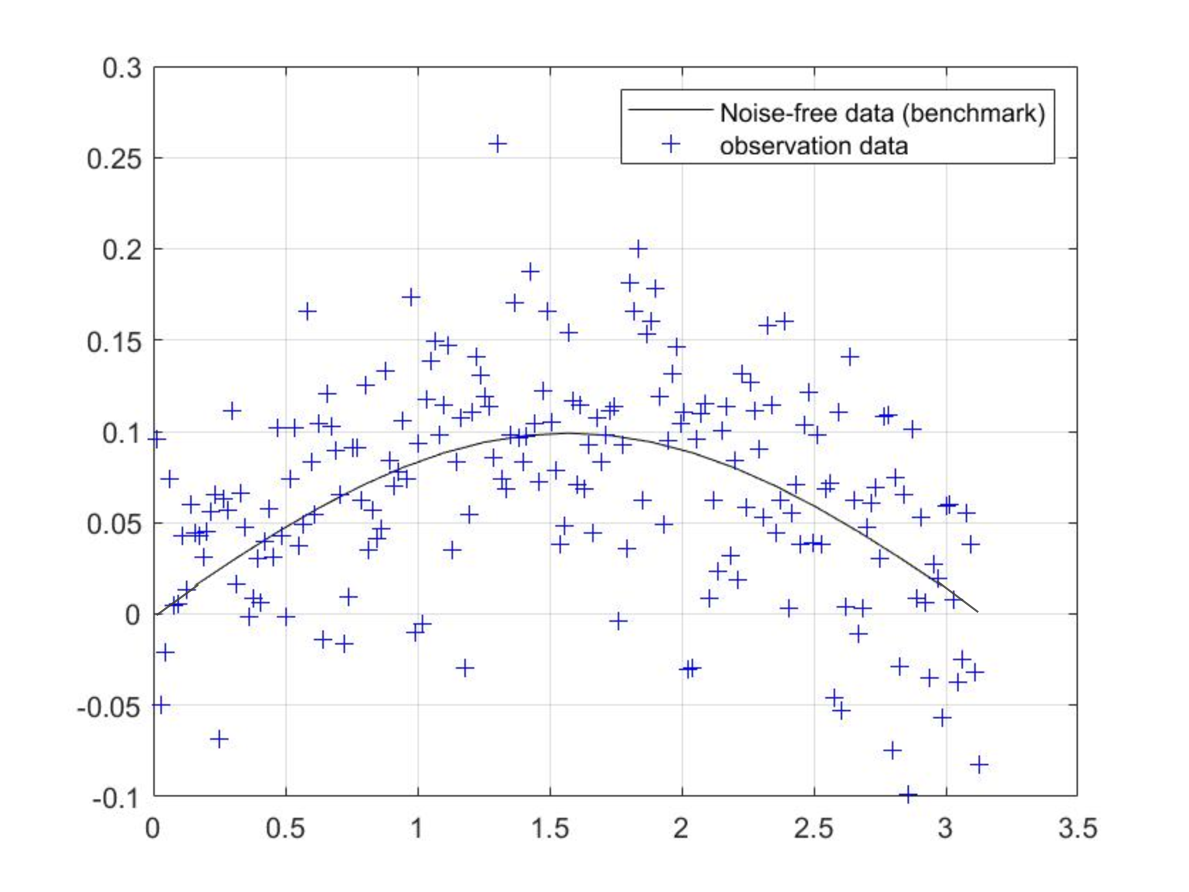}
		\caption{Observation.}
	\end{minipage}
        \begin{minipage}{0.33\linewidth}
		\centering
		\setlength{\abovecaptionskip}{0.28cm}
		\includegraphics[width=1\linewidth]{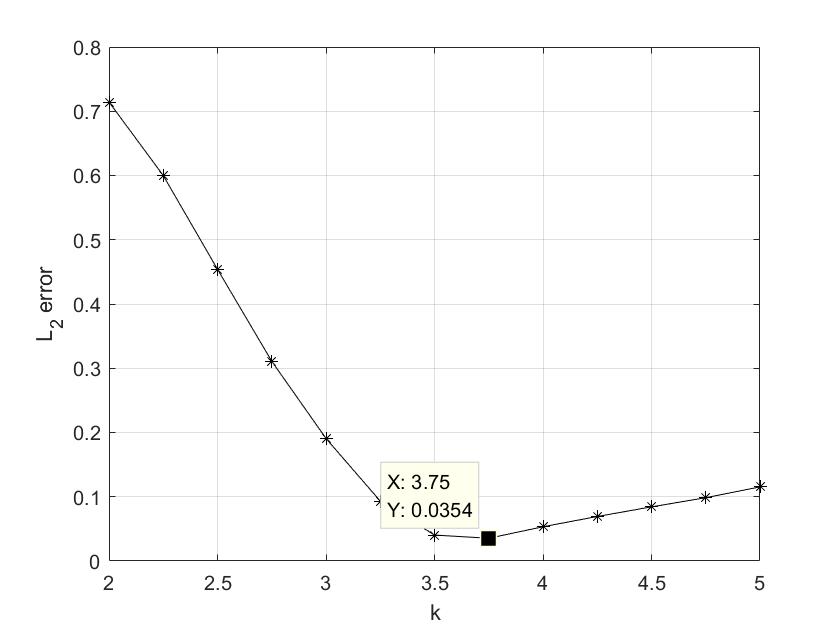}
		\caption{$L^2$ error.}
	  \end{minipage}
        \begin{minipage}{0.33\linewidth}
		\centering
		\setlength{\abovecaptionskip}{0.28cm}
		\includegraphics[width=1\linewidth]{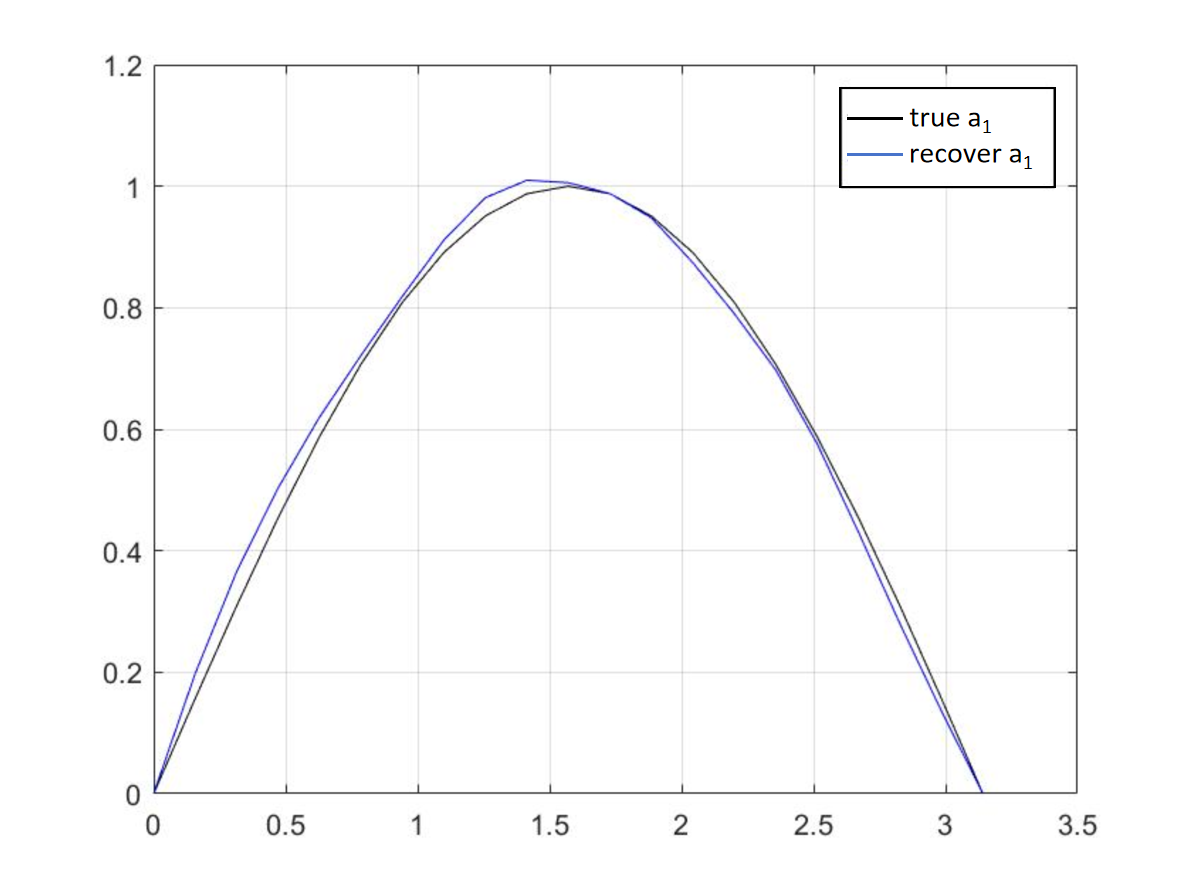}
		\caption{$\rho_{199}=10^{-3.75}$.}
		\label{figure9}
	  \end{minipage}
    
\end{figure}

\begin{figure}[!htb]
        \begin{minipage}{0.49\linewidth}
		\centering
		\setlength{\abovecaptionskip}{0.28cm}
		\includegraphics[width=1\linewidth]{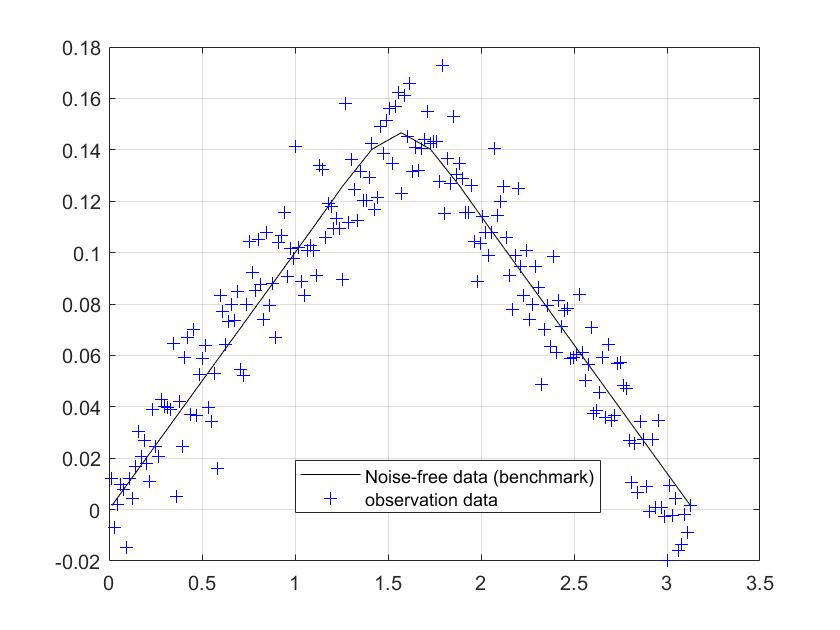}
		\caption{Observation.}
	\end{minipage}
        \begin{minipage}{0.49\linewidth}
		\centering
		\setlength{\abovecaptionskip}{0.28cm}
		\includegraphics[width=1\linewidth]{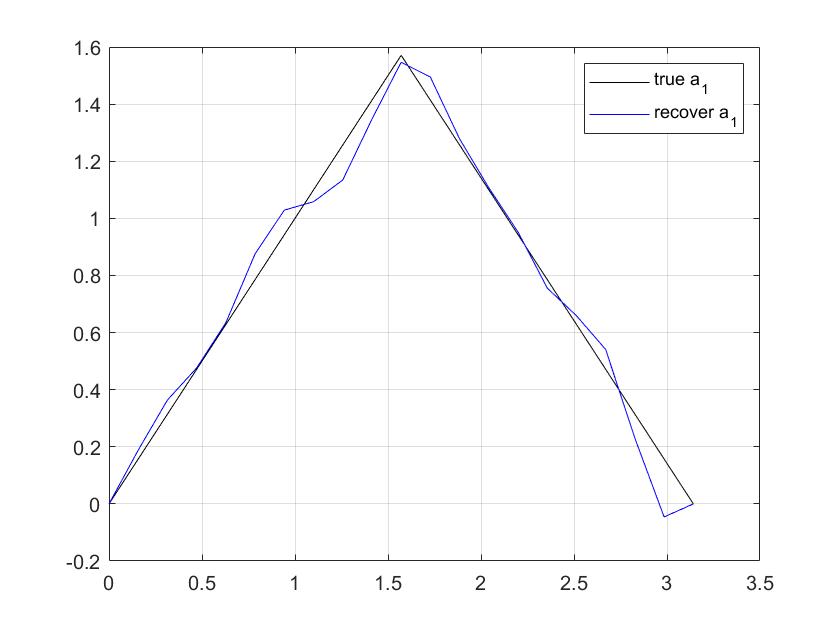}
		\caption{$\rho_{199}=3\times10^{-6}$.}
		\label{figure11}
	  \end{minipage}
    
\end{figure}

\begin{table}[!ht]
\caption{Scheme (\ref{eq-gov-trans}) for Example \ref{ex3} with $\alpha=1.25$.}
\renewcommand{\arraystretch}{1}
\def\temptablewidth{1\textwidth}
\begin{center}
 \begin{tabular*}{\temptablewidth}{@{\extracolsep{\fill}}lcccccc}\hline
 $N$ & $r=1$   &   &$r_{opt}=\frac{4-\alpha}{2-\alpha}$      &     &$r=\frac{4-\alpha}{\alpha}$  &    \\ 
 \cline{2-3}\cline{4-5}\cline{6-7}
 & $e_{L^{2}}(N)$         &Order     &$e_{L^{2}}(N)$    &Order    &$e_{L^{2}}(N)$ &Order     \\ \hline
 20 &2.2836e-03     &-      &1.0643e-03   &-       &9.0033e-04  &-  \\ 
 40 &1.3385e-03     &0.771  &4.2493e-04   &1.325   &3.6861e-04  &1.288    \\ 
 80 &7.6092e-04     &0.815  &1.6494e-04   &1.365   &1.4657e-04  &1.331 \\ 
 160 &4.1144e-04    &0.887  &6.1987e-05   &1.412   &5.6218e-05  &1.383    \\ 
 Optimal order  &\multicolumn{6}{c}{1.375} \\ \hline
\end{tabular*}
\end{center}
\end{table}

\begin{table}[!ht]
\caption{Scheme (\ref{eq-gov-trans}) for Example \ref{ex3} with $\alpha=1.75$.}
\renewcommand{\arraystretch}{1}
\def\temptablewidth{1\textwidth}
\begin{center}
 \begin{tabular*}{\temptablewidth}{@{\extracolsep{\fill}}lcccccc}\hline
 $N$ & $r=1$   &   &$r_{opt}=\frac{4-\alpha}{2-\alpha}$      &     &$r=\frac{4-\alpha}{\alpha}$  &    \\ 
 \cline{2-3}\cline{4-5}\cline{6-7}
 & $e_{L^{2}}(N)$         &Order     &$e_{L^{2}}(N)$    &Order    &$e_{L^{2}}(N)$ &Order     \\ \hline
 20 &1.4366e-02     &-      &8.7758e-03   &-       &1.4216e-02  &-  \\ 
 40 &1.1009e-02     &0.384  &4.3576e-03   &1.010   &1.0541e-02  &0.431    \\ 
 80 &8.0645e-03     &0.449  &2.0575e-03   &1.083   &7.4764e-03  &0.496 \\ 
 160 &5.5304e-03    &0.544  &9.4409e-04   &1.124   &4.9702e-03  &0.589    \\ 
 Optimal order  &\multicolumn{6}{c}{1.125} \\ \hline
\end{tabular*}
\end{center}
\end{table}

\begin{figure}[!htb]
        \begin{minipage}{0.49\linewidth}
		\centering
		\setlength{\abovecaptionskip}{0.28cm}
		\includegraphics[width=1\linewidth]{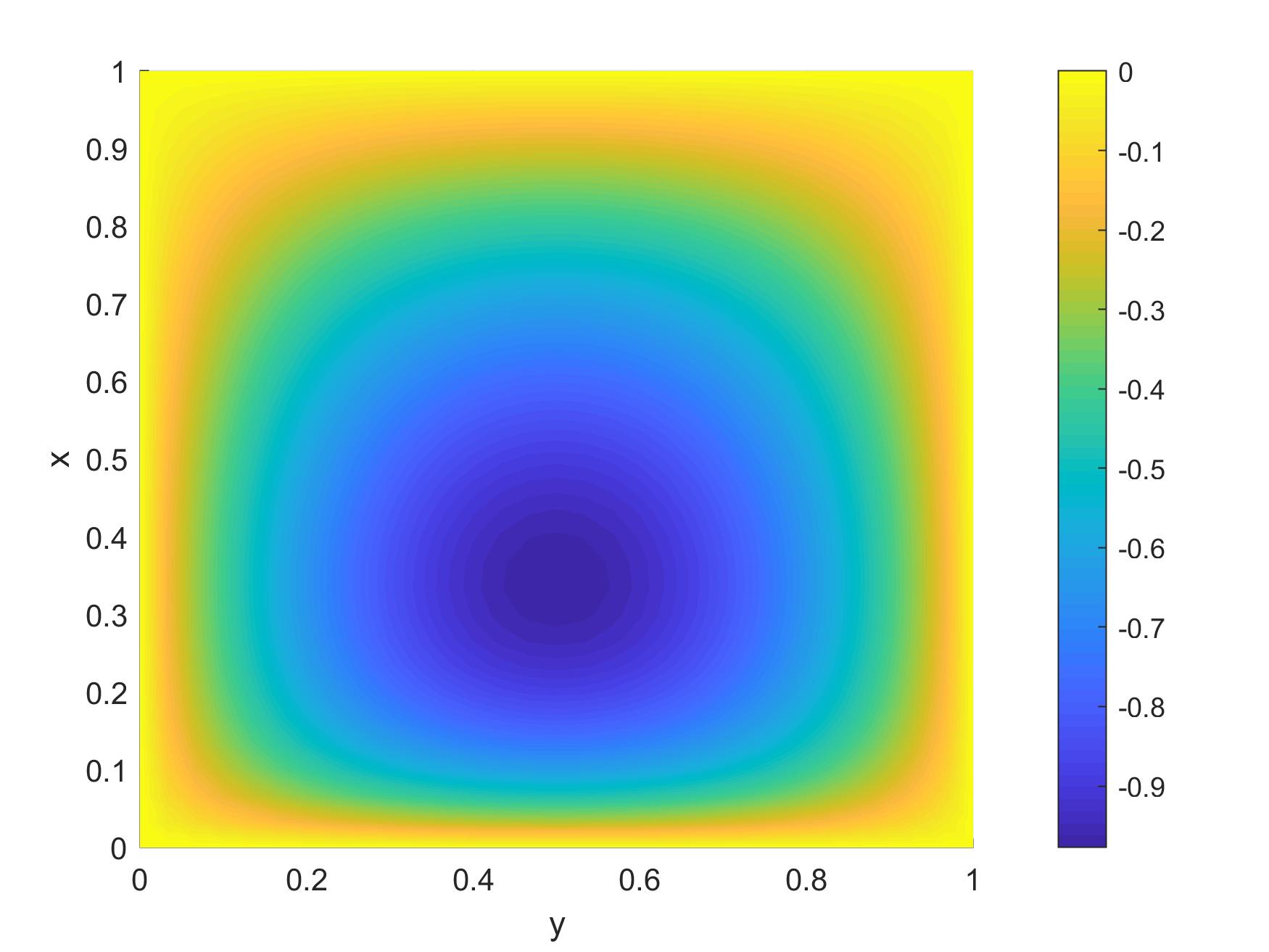}
		\caption{Exact $a_1$.}
		\label{figure12}
	\end{minipage}
        \begin{minipage}{0.49\linewidth}
		\centering
		\setlength{\abovecaptionskip}{0.28cm}
		\includegraphics[width=1\linewidth]{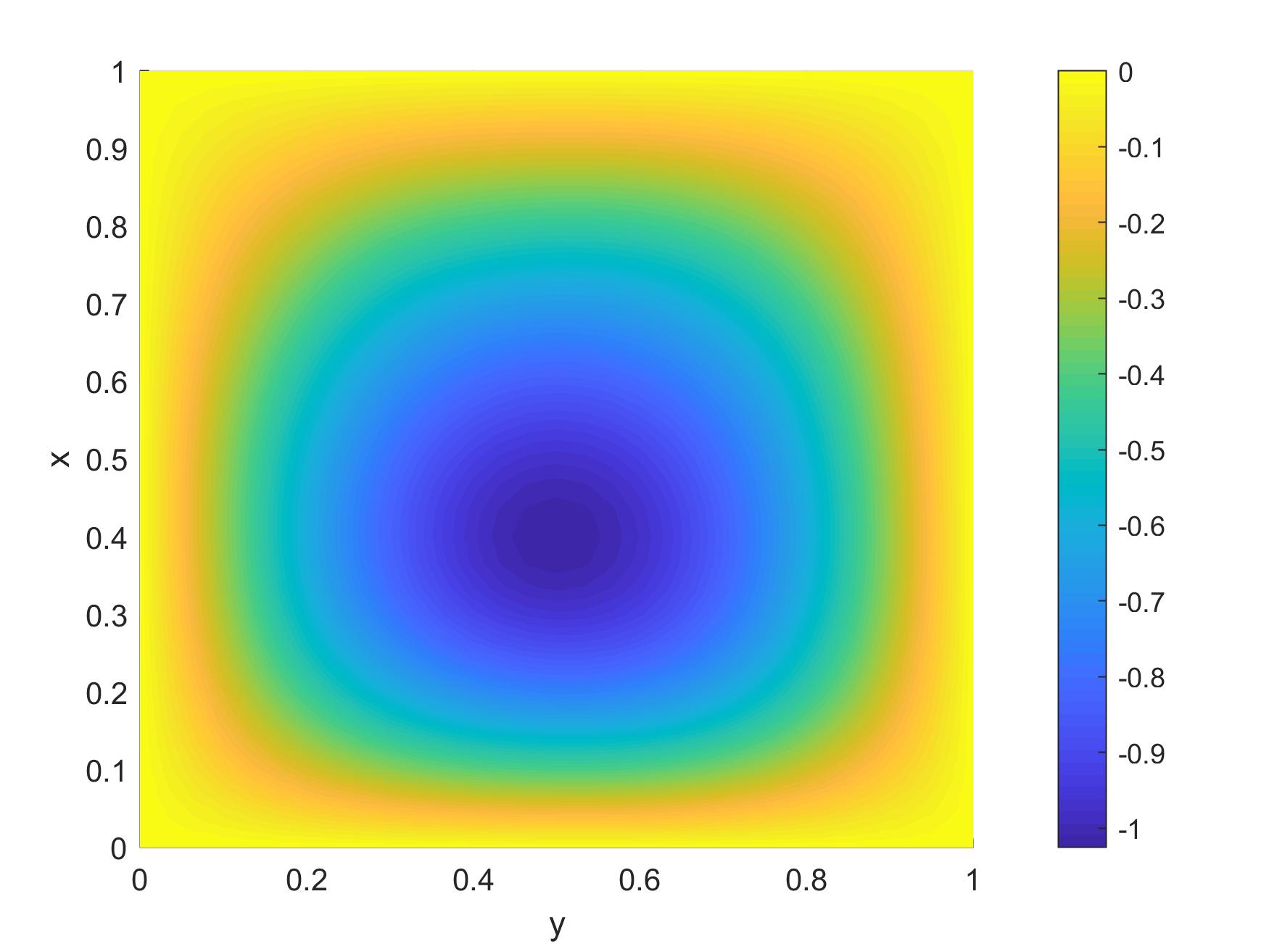}
		\caption{$\rho_{841}=10^{-5}$.}
		\label{figure13}
	  \end{minipage}
    
\end{figure}

\begin{figure}[!htb]
        \begin{minipage}{0.49\linewidth}
		\centering
		\setlength{\abovecaptionskip}{0.28cm}
		\includegraphics[width=1\linewidth]{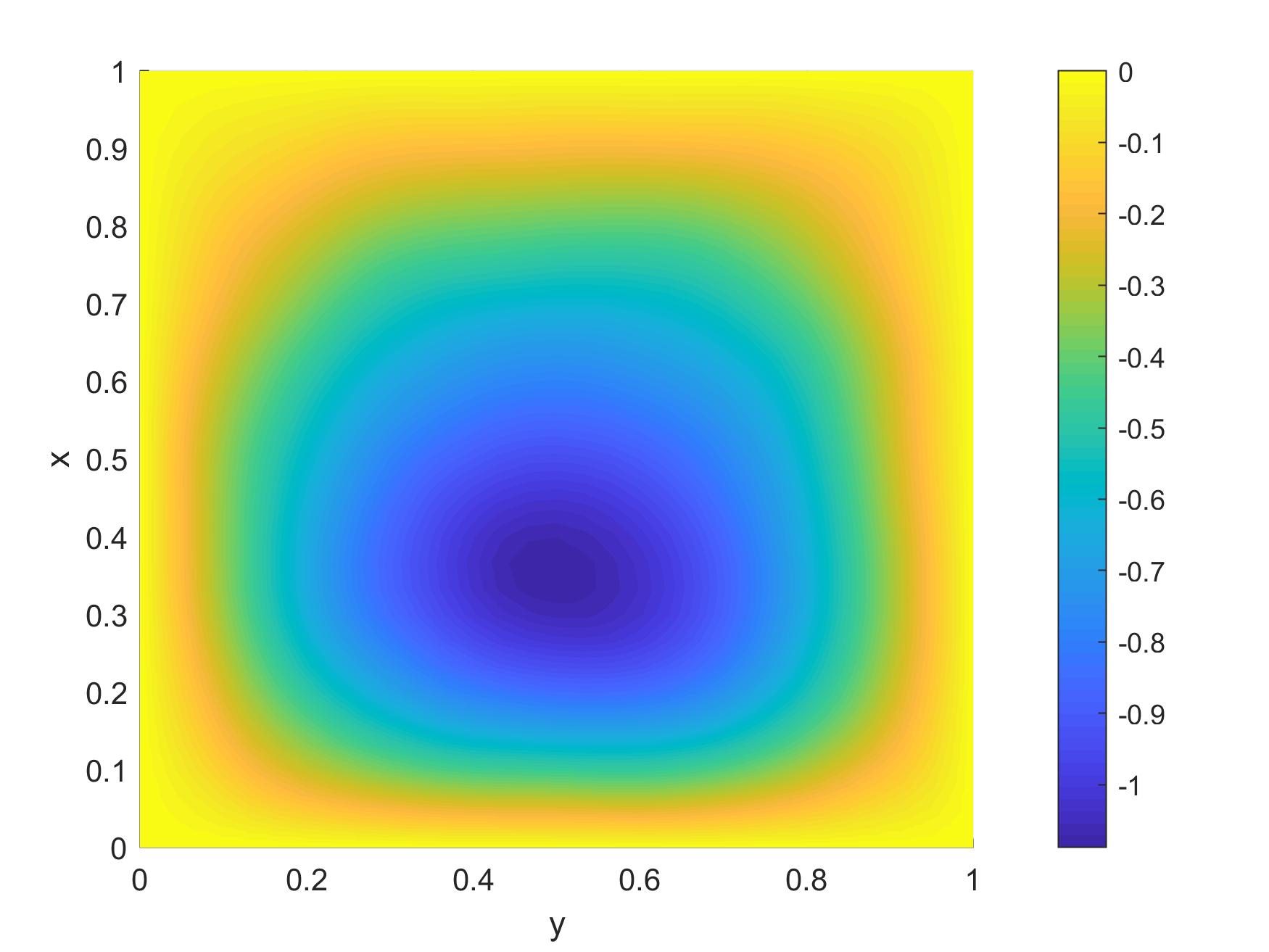}
		\caption{$\rho_{841}=2\times10^{-6}$.}
		\label{figure14}
	\end{minipage}
        \begin{minipage}{0.49\linewidth}
		\centering
		\setlength{\abovecaptionskip}{0.28cm}
		\includegraphics[width=1\linewidth]{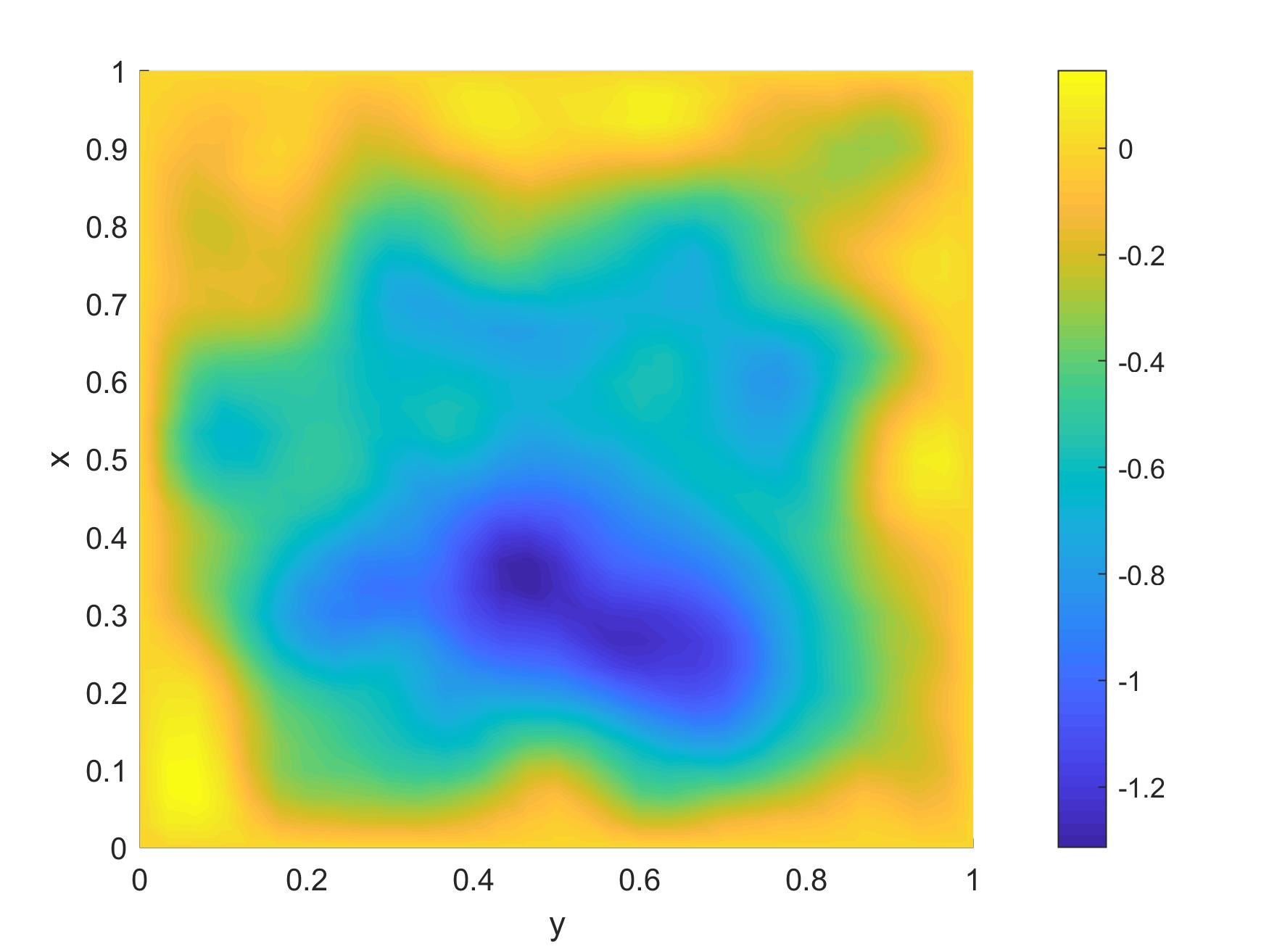}
		\caption{$\rho_{841}=10^{-7}$.}
		\label{figure15}
	  \end{minipage}
    
\end{figure}

\section{Concluding remarks}
We develop a numerical framework for fractional wave equations under the lower regularity assumptions.
The optimal convergence of it is guaranteed by choosing suitable time grids parameter for smooth and nonsmooth solutions. In the numerical framework, a  scattered point measurement-based regularization method is used to solve backward problems with uncertain data.
The optimal error estimates of stochastic convergence not only balance discretization errors, the noise, and the number of observation points, but also propose an a priori choice of regularization parameters. Despite the presence
of large observation errors, we can still obtain more precise inversion results by increasing the number of observation data, even for initial functions with singularity points.

\section*{Declarations}
On behalf of all authors, the corresponding author states that there is no conflict of interest. No datasets were generated or analyzed during the current study.


\end{document}